\documentclass[a4paper,twoside,english]{amsart}
\usepackage{graphicx}  
\usepackage{amsmath,amssymb,amsthm, amscd}
\usepackage{amssymb,fge}
\usepackage{tikz-cd}
\usepackage[T1]{fontenc}
\usepackage[latin9]{inputenc}
\usepackage{babel}
\usepackage{verbatim}
\makeatletter
\let\@wraptoccontribs\wraptoccontribs
\makeatother

\usepackage{lipsum}

\usepackage{color}

\usepackage[left=3.2cm,right=3.5cm,top=3cm,bottom=3cm]{geometry}
\usepackage{indentfirst}           % to indent the first paragraph of a section
\usepackage{underscore}
\usepackage{enumitem}
\usepackage{relsize}
\usepackage{varwidth}
\usepackage{mathtools} 
\usepackage{hyperref}
\hypersetup{
    colorlinks=true,
    linkcolor=blue,
    filecolor=magenta,      
    urlcolor=cyan,
}
\usepackage{amssymb}
\usepackage[all]{xy}
\makeatletter
\renewcommand*\env@matrix[1][*\c@MaxMatrixCols c]{%
  \hskip -\arraycolsep
  \let\@ifnextchar\new@ifnextchar
  \array{#1}}
\makeatother

\def\Gal{\mathop{\rm Gal}\nolimits}

\def\CVD{{\hfill\hfil{\lower 2pt\hbox{\vrule\vbox to 7pt
{\hrule width  5pt\varphifill\hrule}\varphirule}}}\par}

\DeclareMathOperator{\PrePer}{PrePer}

\DeclareMathOperator{\Per}{Per}

\newcommand{\mysetminus}{\mathbin{\fgebackslash}}

\newenvironment{Aut}{{\rm Aut}}{}

\newtheorem{theorem}{Theorem}[section]
\newtheorem{lemma}[theorem]{Lemma}
\newtheorem{proposition}[theorem]{Proposition}
\newtheorem{proposition-definition}[theorem]{Proposition-Definition}

\theoremstyle{definition}

\theoremstyle{remark}
\newtheorem{remark}{Remark}

\theoremstyle{theorem}

\theoremstyle{remark}

\title[Irreducible Polynomials in semigroups]{Irreducible polynomials in quadratic semigroups}    
\begin{document}
\author{Wade Hindes, Reiyah Jacobs, and Peter Ye}  
\maketitle

\begin{abstract} We construct many irreducible polynomials within semigroups generated by sets of the form $S=\{x^2+c_1,\dots,x^2+c_s\}$ under composition.  
\end{abstract} 
\section{Introduction}
Let $K$ be a field and let $S=\{f_1,\dots, f_s\}$ be a set of polynomials with $f_i\in K[x]$. Then one can form the semigroup $M_S$ generated by $S$ under composition, i.e., the set of all polynomials of the form $f=\theta_1\circ\dots\circ\theta_n$ for some $n\geq1$ and some $\theta_i\in S$. These semigroups have gained increased interest in both dynamical systems and number theory in recent years, as they naturally generalize the process of iterating a single function; see \cite{Tits,Cabrera,DoyleFiliTobin,Me:1stCount,Mello,Pakovich1,Pakovich2} for results relating the structure of $M_S$ to preperiodic points, invariant probability measures, and height functions. In addition and of particular interest for this paper, the semigroup $M_S$ can be used to construct arboreal representations of Galois groups, i.e., Galois groups that arise as subgroups of the automorphism group of a tree; see \cite{Odoni,Cubics,FeragutiPagano,Jones, Stoll} for results in the case of iterating a single function and see \cite{MeVivRafe,Me:LeftRightTotal} for recent results in the more general situation. However, nearly all of the known results on arboreal representations rely on the underlying polynomials in the semigroup being irreducible (a natural assumption in Galois theory in general). Nevertheless, this can be quite tricky to prove in the general semigroup situation. In this paper, we make some progress on this problem in the case when $M_S$ is generated by quadratic polynomials of the form $x^2+c$ with integral coefficients. In particular, we show that if $S$ contains at least two such irreducible quadratics, then many polynomials in $M_S$ are irreducible in $\mathbb{Q}[x]$. For simplicity here and throughout (unless explicitly stated otherwise), irreducible means irreducible in $\mathbb{Q}[x]$. Moreover, the superscript $N$ below in $\phi^N$ denotes iteration.

 \vspace{.15cm}  

\begin{theorem}\label{thm:main} Let $c_1,\dots, c_s$ be distinct integers, let $S=\{x^2+c_1,\dots,x^2+c_s\}$, and let $M_S$ be the semigroup generated by $S$ under composition. Moreover, assume that $S$ contains at least two irreducible polynomials. Then there exist $N\geq2$ and $\phi_i,\phi_j\in S$ such that \vspace{.05cm} 
\[\{\phi_i^N\circ\phi_j\circ \phi_i\circ F\,:\, F\in M_S\} \vspace{.05cm} \]
is a set of irreducible polynomials.\vspace{.25cm}    
\end{theorem}
%\begin{remark} If $K$ is a field of characteristic zero and $S=\{x^2+c_1,\dots, x^2+c_s\}$ for some $c_i\in K^*$, then $M_S$ is free \cite[Theorem 3.1]{Me:free}. Hence, we may define the length of $f=\theta_1\circ\dots\circ\theta_n\in M_S$ for some $\theta_i\in S$ to be $\ell(f)=n$. In particular, given a subset $A\subseteq M_S$, the function $\mu(A)=\sum_{f\in A}(2s)^{-\ell(f)}$ defines a probability measure on $M_S$. Moreover, it follows from Theorem \ref{thm:main} that if the $c_i$ are integral and $S$ contains at least $3$ irreducible polynomials, then the $\mu$-measure of the set of irreducible polynomials in $M_S$ is at least $(2s)^{-(N+1)}$ for some $N$ (thus positive).     
%\end{remark} 
\begin{remark} Let $\nu$ be probability measure on $S=\{x^2+c_1,\dots, x^2+c_s\}$, and and assume that $\nu(\phi)>0$ for all $\phi\in S$ (so that no maps have probability zero). Now extend $\nu$ to a probability measure $\bar{\nu}$ on the set $\Phi_S:=\{\gamma=(\theta_1,\theta_2, \dots)\;\vert\,\theta_n\in S\}$ of infinite sequences of elements of $S$ via the product measure. Then one can attach to each infinite sequence $\gamma$ an arboreal representation $\rho_\gamma: \Gal(\overline{\mathbb{Q}}/\mathbb{Q})\rightarrow\Aut (T_{\gamma,0})$, where $T_{\gamma,0}$ is the infinite, regular, $2$-ary rooted tree of preimages of $0$ under the sequence of maps defining $\gamma$; see \cite{Feraguti,MeVivRafe} for details. Then, it follows from Theorem \ref{thm:main} above that if the $c_i$ are integral and $S$ contains at least $2$ irreducible maps, then an arboreal representation acts transitively with \emph{positive probability},
\[\bar{\nu}\big(\{\gamma\in\Phi_S\,:\,\text{$\rho_\gamma\big( \Gal(\overline{\mathbb{Q}}/\mathbb{Q})\big)$ acts transitively on $T_{\gamma,0}$} \}\big)>0.\]
Equivalently, a random sequence $\gamma\in\Phi_S$ is \emph{stable} over $\mathbb{Q}$ (i.e., for a sequence $\gamma=(\theta_m)_{m\geq1}$, the polynomials $\theta_1\circ\dots \circ\theta_m$ are irreducible for all $m\geq1$) with positive probability \cite{Feraguti,MeVivRafe}.             
\end{remark} 
\begin{remark} It may be that having one irreducible polynomial is enough to force stability with positive probability. However, if one uses the the standard irreducibility test for iterative quadratic sequences (looking for squares in the adjusted critical orbit, see Proposition \ref{prop:stability} below), then the theorem above is optimal: when $S=\{x^2-4,x^2-12\}$,  \emph{every} infinite sequence in $S$ except for the constant sequence $(x^2-12, x^2-12,\dots)$, has a square somewhere in its adjusted critical orbit. Therefore, the stability test in Proposition \ref{prop:stability} fails with probability one in this case. \vspace{.25cm}    
\end{remark} 

The main tool we use to establish the existence of many irreducible polynomials in semigroups generated by quadratics is the following classification theorem, which roughly says that if a pair of maps $(\phi_1,\phi_2)=(x^2+c_1,x^2+c_2)$ has the property that there exist large iterates $N_1$ and $N_2$ and points $b_1,b_2\in\mathbb{Z}$ such that $\phi^{N_1}(\phi_2(b_1))$ and $\phi_2^{N_2}(\phi_1(b_2))$ are \emph{both} rational squares, then $(\phi_1,\phi_2)$ is of a very special form; see Lemma \ref{lem:Per-square} and Proposition \ref{prop:nosquares} to see why this formulation is equivalent to the statement below and see Proposition \ref{prop:stability} for the salience of squares in this context. In what follows, given a polynomial $\phi\in\mathbb{Q}[x]$, we let $\Per(\phi,\mathbb{Q})$ denote the set of rational periodic points of $\phi$, i.e, the set of all $P\in\mathbb{Q}$ such that $\phi^n(P)=P$ for some $n\geq1$. Likewise, we let $\PrePer(\phi,\mathbb{Q})$ denote the set of rational preperiodic points for $\phi$, i.e., the set of all $P\in\mathbb{Q}$ such that $\phi^n(P)\in\Per(\phi,\mathbb{Q})$ for some $n\geq1$. Then we have the following classification theorem leading to irreducible polynomials in semigroups: \vspace{.2cm}          

\begin{theorem}\label{thm:classification} Let $\phi_1=x^2+c_1$ and $\phi_2=x^2+c_2$ for some distinct integers $c_i\in\mathbb{Z}$. Moreover, assume the following conditions hold: \vspace{.2cm}  
\begin{enumerate} 
\item[\textup{(1)}] Both $\phi_1$ and $\phi_2$ have periodic points that are rational squares. \vspace{.2cm}  
\item[\textup{(2)}] Both $\phi_1(\mathbb{Z})\cap\PrePer(\phi_2,\mathbb{Q})$ and $\phi_2(\mathbb{Z})\cap\PrePer(\phi_1,\mathbb{Q})$ are non-empty. \vspace{.2cm}    
\end{enumerate} 
Then, up to reordering, either 
\[(c_1,c_2)=(-1,-3)\;\;\;\text{or}\;\;\; (c_1,c_2)=(s^2-s^4,-1-s^2-s^4)\] 
for some $s\in\mathbb{Z}$. \vspace{.2cm}   
\end{theorem}

\begin{remark} We call a pair of polynomials $\phi_1$ and $\phi_2$ satisfying properties (1) and (2) in Theorem \ref{thm:classification} above \emph{exceptional pairs}. The proof of the classification of exceptional pairs may be found in section \ref{sec:exceptionalpairs}. The argument involves mixing dynamical techniques (e.g., explicit characterizations of periodic points for quadratic polynomials) with arithmetic techniques for finding all integral and rational points on many curves and surfaces (via Runge's method, analyzing Mordell Weil groups of Jacobians, etc.).  \vspace{.2cm}       
\end{remark}  

\textbf{Acknowledgements:} We thank Rafe Jones and Michael Stoll for conversations related to the work in this paper. We also thank the Mathworks summer program at Texas State University for supporting this collaboration. Finally, the first author thanks MSRI for their support during the spring of 2022, when this project began.  

\medskip

\section{Background: arithmetic with quadratic sequences}\label{sec:background}

The first thing we need is a convenient irreducibility test for iterates. In what follows, let $K$ be a field of characteristic not $2$, let $S=\{x^2+c_1,\dots, x^2+c_s\}$ be a set of quadratic polynomials (with common critical point zero) for some $c_i\in K$, and let $M_S$ be the semigroup generated by $S$ under composition. Then given a polynomial $f=\theta_1\circ\dots\circ\theta_n\in M_S$ for some $\theta_i\in S$, it suffices to check that the corresponding \emph{adjusted critical orbit}, \vspace{.05cm}  
\[-\theta_1(0),\; \theta_1(\theta_2(0))\,,\;\theta_1(\theta_2(\theta_3(0)))\,,\; \dots\;,\;\theta_1(\theta_2(\dots(\theta_n(0)))), \vspace{.05cm}\] 
contains no squares in $K$; see \cite[Proposition 6.3]{Me:LeftRightTotal}. \vspace{.1cm} 

\begin{proposition}\label{prop:stability} Let $K$ be a field of characteristic not $2$, let $S=\{x^2+c_1,\dots, x^2+c_s\}$ for some $c_i\in K$, and suppose that $f=\theta_1\circ\dots\circ\theta_n\in M_S$ satisfies the following properties: \vspace{.15cm} 
\begin{enumerate} 
\item[\textup{(1)}] $-\theta_1(0)$ is not a square in $K$, \vspace{.15cm} 
\item[\textup{(2)}] $\theta_1(\theta_2(\dots(\theta_m(0))))$ is not a square in $K$ for all $1\leq m\leq n$. \vspace{.15cm}
\end{enumerate}
Then $f$ is irreducible in $K[x]$.   \vspace{.05cm} 
\end{proposition}

Hence, our strategy to construct irreducible polynomials $f\in M_S$ will be to construct $f$'s whose adjusted critical orbit (defined above) contains no squares. In particular, we focus on polynomials of the form $f=\phi^N\circ g$ for some $\phi\in S$ and some $g\in M_S$ (i.e., the first few letters of $f$ are the same); in this case, one can apply techniques and tools from arithmetic dynamics (applied to the single polynomial $\phi$) to gain information about the possible images of $f$. For example, for generic $\phi$ and large $N$, if $\phi^N(a)$ is a square, then $a$ must be a preperiodic point for $\phi$; see Lemma \ref{lem:Per-square} below. In particular, if $N$ is large enough, then $\phi^N(a)$ must be a \emph{periodic point} of $\phi$. Hence, $\phi$ has a rational (or integral) periodic point that is a square. This property should be quite restrictive, and indeed when $\phi=x^2+c$ for some $c\in\mathbb{Z}$, then we can characterize the possible values of $c$ explicitly.      

\begin{proposition}\label{prop:portraits} Let $\phi=x^2+c$ for some $c\in\mathbb{Z}$. If $\Per(\phi,\mathbb{Q})$ contains a rational square, then one of the following statements must be true: \vspace{.2cm} 
\begin{enumerate}
\item[\textup{(1)}]  $c=s^2-s^4$ for some $s\in\mathbb{Z}$ and $\PrePer(\phi,\mathbb{Q})=\{\pm{s^2},\pm{(1-s^2)}\}$.  \vspace{.25cm}   
\item[\textup{(2)}] $c=-1-s^2-s^4$ for some $s\in\mathbb{Z}$ and $\PrePer(\phi,\mathbb{Q})=\{\pm{s^2},\pm{(1+s^2)}\}$.   
\end{enumerate}   
\end{proposition}
\begin{proof} First if $c=0$, then $\phi(x)=x^2$ and $\PrePer(x^2,\mathbb{Q})=\{0,\pm{1}\}$. In particular, $\phi$ fits the form of case (1) above with $s=0$ (modulo the fact that $\pm{s^2}$ is a single point. Likewise if $c=-1$, then $\phi(x)=x^2-1$ and $\PrePer(x^2-1,\mathbb{Q})=\{0,\pm{1}\}$. In particular, $\phi$ fits the form of case (2) above with $s=0$ (again, modulo the fact that $\pm{s^2}$ is a single point. Therefore, from this point forward we assume that $c\neq0,-1$. Now, it is known that if $f\in\mathbb{Z}[x]$ is a monic polynomial and $P\in\Per(\phi,\mathbb{Q})$, then $P\in\mathbb{Z}$ and $P$ has period $1$, $2$ or $4$; see, for instance,  \cite[Exercise 2.20]{Silverman}, a consequence of examining periods modulo small primes. On the other hand, \cite[Theorem 4]{Morton} implies that $x^2+c\in\mathbb{Q}[x]$ have no rational points of period $4$. Combining these results, if $\phi=x^2+c$ for some $c\in\mathbb{Z}$, then any rational periodic point of $\phi$ is integral and has period at most $2$. Therefore, if $\Per(\phi,\mathbb{Q})$ contains a rational square, then $P=s^2$ for some $s\in\mathbb{Z}$ and either $P$ is a fixed point or a point of exact period $2$. From here, we proceed in cases: \\[5pt] 
\textbf{Case(1):} Suppose that $P=s^2$ is a fixed point. Then $(s^2)^2+c=s^2$ and so $c=s^2-s^4$ as in statement (1) of Proposition \ref{prop:portraits}. On the other hand, any other fixed point $x$ satisfies $0=x^2+s^2-s^4-x=(x-s^2)(x-(1-s^2))$, so that the complete list of fixed points of $\phi$ are $s^2$ and $1-s^2$ (which are distinct for non-zero $s$). Now if $\phi$ had an a point $t$ of exact period $2$, then $t$ is an integer and   
\[0=\phi^2(t)-t=(t^2+c)^2+c-t=(t^2-t+c)(t^2+t+c+1).\]
However, the left factor corresponds to a fixed point, so that $s^2-s^4=c=-1-t-t^2$. However, this equation has no solutions $(s,t)\in(\mathbb{Z}/4\mathbb{Z})^2$. Therefore, no such $t\in\mathbb{Z}$ exists and the only rational periodic points of $\phi$ are $s^2$ and $1-s^2$. Next, to determine the rational preperiodic points of $\phi$, we examine the preimages of $s^2$ and $1-s^2$. It is clear that $\phi^{-1}(s^2)=\{\pm{s^2}\}$. Moreover, if $y\in\mathbb{Q}$ and $y^2+s^2-s^4=\phi(y)=-s^2$, then $y\in\mathbb{Z}$ and 
\[y^2=s^4+2s^2=s^2(s^2+2).\]
However, $s\neq0$ (since we assume that $c\neq0$) so that $s^2+2=z^2$ for some $z\in\mathbb{Z}$. But this equation has no solutions mod $4$. Hence, $\{\pm{s^2}\}$ is the complete set of rational preimages of $s^2$ under iteterates of $\phi$. Now for the preimages of $1-s^2$. Note first that $\phi^{-1}(1-s^2)=\{\pm{(1-s^2)}\}$. On the other hand, if $y\in\mathbb{Q}$ and $y^2+s^2-s^4=\phi(y)=-(1-s^2)$, then $y\in\mathbb{Z}$ and $y^2=s^4-1$. However, this equation has no solutions: $(y-s^2)(y+s^2)=-1$ implies $y-s^2=1$ and $y+s^2=-1$ or $y-s^2=-1$ and $y+s^2=1$. But in either case, adding the equations implies $y=0$. But then $s^2=-1$, a contradiction. Therefore, $\{\pm{(1-s^2)}\}$ is the complete set of rational preimages of $1-s^2$ under iteterates of $\phi$. Hence, $\PrePer(\phi,\mathbb{Q})=\{\pm{s^2},\pm{(1-s^2)}\}$ as claimed.      
\\[5pt] 
\textbf{Case(2):} Suppose that $P=s^2$ is a point of exact period $2$. Then 
\[0=\phi^2(s^2)-s^2=(c + s^4 - s^2)(c + s^4 + s^2 + 1).\]
However, the left factor is non-zero since $P$ is not a fixed point. Thus, $c=-1-s^2-s^4$ as in statement (2) of Proposition \ref{prop:portraits}. On the other hand, if $\phi$ has another point $t$ if period $2$, then (repeating the calculation above) implies $-1-t-t^2=c=-1-s^2-s^4$ so that $0=(t-s^2)(t+s^2+1)$. Hence, the only rational points of exact period $2$ are $s^2$ and $-(s^2+1)$. In fact, $\phi$ maps one to the other and vice versa. Likewise, if $\phi$ has a rational fixed point $t$, then $t$ is integral and $t-t^2=c=-1-s^2-s^4$. However, this equation has no solutions $(s,t)\in(\mathbb{Z}/4\mathbb{Z})^2$. Therefore, no such $t\in\mathbb{Z}$ exists and the only rational periodic points of $\phi$ are $s^2$ and $-(s^2+1)$. Next, to determine the rational preperiodic points of $\phi$, we examine the preimages of $s^2$ and $-(s^2+1)$. It is easy to check that $\phi^{-1}(s^2)=\{\pm{(s^2+1)}\}$ and $\phi^{-1}(-(s^2+1))=\{\pm{s^2}\}$. On the other hand, if $t\in\mathbb{Q}$ is such that $\phi(t)=s^2+1$, then $t$ is integral and 
\[t^2=s^4+2s^2+2.\]   
However, the hyperelliptic curve $C$ with affine model $t^2=s^4+2s^2+2$ is birational over $\mathbb{Q}$ to the elliptic curve $E: y^2 = x^3 - 4x^2 - 4x$. Moreover, we compute with Magma that $E(\mathbb{Q})\cong\mathbb{Z}/2\mathbb{Z}$. Therefore $\#C(\mathbb{Q})=2$, corresponding to the two points at infinity. Hence, there are no $t\in\mathbb{Q}$ for which $\phi(t)=s^2+1$. Likewise, if $t\in\mathbb{Q}$ is such that $\phi(t)=-s^2$, then $t$ is integral and 
\[(t-s^2)(t+s^2)=1.\]
Therefore, $t-s^2=1=t+s^2$ or $t-s^2=-1=t+s^2$, which in either case implies that $s=0$. However, this contradicts our assumption that $c\neq-1$; this case was handled at the beginning of the proof. Hence, we deduce that $\PrePer(\phi,\mathbb{Q})=\{\pm{s^2},\pm{(1+s^2)}\}$ as claimed. \vspace{.25cm}               
\end{proof}  

\section{Proof of Main Theorem}
In this section we prove our main result, Theorem \ref{thm:main}, on the existence of many irreducible polynomials within semigroups generated by sets of the form $S=\{x^2+c_1,\dots, x^2+c_s\}$. To do this, we assume Theorem \ref{thm:classification} on the classification of exceptional pairs and leave the proof of this fact to Section \ref{sec:exceptionalpairs}. First, as  Proposition \ref{prop:stability} makes clear, to show that a polynomial $f\in M_S$ is irreducible, it suffices to show that an associated adjusted critical orbit contains no squares. As a first step in this direction, we note that if $\phi=x^2+c\in\mathbb{Z}$ and $a\in\mathbb{Z}$ are such that $\phi^N(a)$ is a square for some large $N$ (depending on $\phi$), then $a$ must be preperiodic for $\phi$ and $\phi$ must have a square periodic point. \vspace{.1cm}    

\begin{lemma}\label{lem:Per-square} Let $\phi=x^2+c$ for some $c\in\mathbb{Z}\mysetminus\{0,-1\}$. Then there exists an iterate $N=N_\phi\geq2$ such that if $\phi^N(a)$ is a square in $\mathbb{Q}$ for some $a\in\mathbb{Z}$, then $a\in\PrePer(\phi,\mathbb{Q})$ and $\phi$ has a rational periodic point that is a square. \vspace{.1cm}         
\end{lemma} 
\begin{proof} First consider the hyperelliptic curve $C_\phi$ with affine model $Y^2=\phi^2(X)$. Then $C$ has genus $1$: the discriminant of $\phi^2$ is a non-zero multiple of $\phi(0)\cdot\phi^2(0)=c\cdot (c^2+c)$, which is non-zero since $c\neq0,-1$. In particular, Siegel's theorem implies that $C$ has only finitely many integral points. Hence, we can choose a positive constant $B$ such that if $(x,y)\in C$ for some $x,y\in\mathbb{Z}$, then $\hat{h}_\phi(x)< B$; here, $h_\phi(x)=\lim h(\phi^n(x))/2^n$ is the usual canonical height function. Likewise, Northcott's theorem implies that the quantity 
\[\hat{h}_{\phi,\mathbb{Q}}^{\min}=\min\{\,\hat{h}_\phi(b)\,:\, b\in\mathbb{Q}\;\text{and}\; \hat{h}_\phi(b)>0\}\]
is strictly positive. Next we enlarge $B$ if necessary to assume that $B\geq\hat{h}_{\phi,\mathbb{Q}}^{\min}$ and define $N_\phi:=\lceil \log_2(B/\hat{h}_{\phi,\mathbb{Q}}^{\min})\rceil+2$. Note that $N_\phi\geq2$ by this assumption on $B$. Now suppose that $a\in\mathbb{Z}$ is such that $\phi^N(a)=y^2$ for some $y\in\mathbb{Q}$. Then $y$ must be an integer since $\mathbb{Z}$ is integrally closed in $\mathbb{Q}$. Hence, $(X,Y)=(\phi^{N-2}(a),y)$ is an integral point on $C$ and
\[2^{N-2}\cdot \hat{h}_\phi(a)=\hat{h}_\phi(\phi^{N-2}(a))< B\]
by the functoriality of canonical heights. Now suppose for a contradiction that $a\not\in\PrePer(\phi,\mathbb{Q})$. Then $\hat{h}_\phi(a)>0$, so that $\hat{h}_{\phi,\mathbb{Q}}^{\min}\leq \hat{h}_\phi(a)$. Hence, the inequality above implies that $N< \log_2(B/\hat{h}_{\phi,\mathbb{Q}}^{\min})+2$, a contradiction. Therefore, $a$ is preperiodic for $\phi$. On the other hand, if $f\in\mathbb{Z}[x]$ is a monic polynomial and $P\in\Per(\phi,\mathbb{Q})$, then $P\in\mathbb{Z}$ and $P$ has period $1$, $2$ or $4$; see, for instance,  \cite[Exercise 2.20]{Silverman}. Moreover, \cite[Theorem 4]{Morton} implies that $x^2+c\in\mathbb{Q}[x]$ have no rational points of period $4$. Combining these results, if $\phi=x^2+c$ for some $c\in\mathbb{Z}$, then any rational periodic point of $\phi$ is integral and has period at most $2$. Next we note that part 6 of \cite[Theorem 3]{Poonen} implies that $\phi^2(a)$ is already periodic (after two iterates, any rational preperiodic point must enter a periodic cycle). Therefore, $\phi^N(a)$ is periodic since $N\geq2$. But $y^2=\phi^N(b)$ by assumption, so that $\phi$ has a rational periodic point that is a square.                         
\end{proof}  
Next, we put the previous result together with our classification of exceptional pairs to show that if a pair of maps $(\phi_1,\phi_2)$ has the property that there exist large iterates $N_1$ and $N_2$ and points $b_1,b_2\in\mathbb{Z}$ such that $\phi^{N_1}(\phi_2(b_1))$ and $\phi_2^{N_2}(\phi_1(b_2))$ are \emph{both} squares, then $(\phi_1,\phi_2)$ is of a very special form.  

\begin{proposition}\label{prop:nosquares} Let $\phi_1=x^2+c_1$ and $\phi_2=x^2+c_2$ for some distinct $c_1,c_2\in\mathbb{Z}\mysetminus\{0,-1\}$. Moreover, assume that the pair $(c_1,c_2)$ is not of the form $(s^2-s^4,-1-s^2-s^4)$ for some $s\in\mathbb{Z}$, up to reordering. Then one of the following statements must be true:\vspace{.15cm} 
\begin{enumerate} 
\item[\textup{(1)}] There is $N\geq2$ such that $\phi_1^N(\phi_2(b))$ is not a rational square for all $b\in\mathbb{Z}$. \vspace{.15cm}
\item[\textup{(2)}] There is $N\geq2$ such that $\phi_2^N(\phi_1(b))$ is not a rational square for all $b\in\mathbb{Z}$. \vspace{.1cm}   
\end{enumerate}   
\end{proposition} 
\begin{proof} Assume that $c_1,c_2\in\mathbb{Z}\mysetminus\{0,-1\}$ are distinct and choose positive integers $N_1$ and $N_2$ as in Lemma \ref{lem:Per-square} for $\phi_1$ and $\phi_2$ respectively. Now suppose that statements (1) and (2) of Proposition \ref{prop:nosquares} are both false. Then there exist $b_1,b_2\in\mathbb{Z}$ such that $\phi_1^{N_1}(\phi_2(b_1))$ and $\phi_2^{N_2}(\phi_1(b_2))$ are both rational squares. Hence, if we set $a_1=\phi_2(b_1)$ and $a_2=\phi_1(b_2)$, then we deduce from Lemma \ref{lem:Per-square} that 
\[a_1\in \phi_2(\mathbb{Z})\cap \PrePer(\phi_1,\mathbb{Q})\;\;\text{and}\;\; a_2\in \phi_1(\mathbb{Z})\cap \PrePer(\phi_2,\mathbb{Q})\]
and both $\phi_1$ and $\phi_2$ have rational periodic points that are squares. But then assumptions (1) and (2) of Theorem \ref{thm:classification} hold. Therefore, after perhaps reordering $c_1$ and $c_2$, we have that \[(c_1,c_2)=(s^2-s^4,-1-s^2-s^4)\] 
for some $s\in\mathbb{Z}$ as desired; the other pair $(-1,-3)$ is ruled out by our assumption that $c_1,c_2\neq0,-1$.        
\end{proof} 

On the other hand, it is still possible to exclude square images in the case of exceptional pairs with a bit more effort. 

\begin{proposition}\label{prop:nosquares-exceptional} Let $(c_1,c_2)=(s^2-s^4,-1-s^2-s^4)$ for some $s\in\mathbb{Z}\mysetminus\{0,\pm{1}\}$ and let $\phi_i=x^2+c_i$. Then there is an $N\geq2$ such that $\phi_1^N\circ\phi_2\circ\phi_1(b)$ is not a rational square for all $b\in\mathbb{Z}$. 
\end{proposition} 
\begin{proof} Let $(c_1,c_2)=(s^2-s^4,-1-s^2-s^4)$ for some $s\in\mathbb{Z}\mysetminus\{0,\pm{1}\}$ and let $\phi_i=x^2+c_i$. In particular, $c_1\neq0,-1$. Now choose $N\geq2$ for $\phi_1$ as in Lemma \ref{lem:Per-square} and suppose that $\phi_1^N\circ\phi_2\circ\phi_1(b)=y^2$ for some $b,y\in\mathbb{Z}$. Then Proposition \ref{prop:portraits} and Lemma \ref{lem:Per-square} together imply that $a=\phi_2(\phi_1(b))\in\PrePer(\phi_1,\mathbb{Q})=\{\pm{s},\pm{(1-s^2)}\}$. However, if $a=\pm{(1-s^2)}$, then $y^2=\phi^N(a)=1-s^2$ and $y^2+s^2=1$. However, this forces $s=0,\pm{1}$, a contradiction. Therefore, $a=\pm{s^2}$. Now let $z=\phi_1(b)\in\mathbb{Z}$ so that  
\[\phi_2(z)=z^2-1-s^2-s^4=\pm{s^2}.\] 
From here we proceed in cases. First if $\phi_2(z)=s^2$, then it must be the case that $z=\pm{(s^2+1)}$ (after all, $\phi_2$ is a quadratic map and both of these distinct choices of $z$ map to $s^2$). Hence, $\pm{(s^2+1)}=z=\phi_1(b)=b^2+s^2-s^4$. However, the equation $-(s^2+1)=b^2+s^2-s^4$ has no solutions modulo $4$. Therefore, $s^2+1=b^2+s^2-s^4$ and  thus $1=(b-s^2)(b+s^2)$. However, this easily implies that $(b,s)=(\pm{1},0)$, a contradiction of our assumption on $s$. Therefore, it must be the case that $\phi_2(z)=-s^2$. But this implies that $z^2-1-s^2-s^4=-s^2$ and thus $(z-s^2)(z+s^2)=1$. However, again this forces $(z,s)=(\pm{1},0)$, a contradiction. In particular, it is not possible for  $\phi_1^N\circ\phi_2\circ\phi_1(b)$ to be a square as claimed.                
\end{proof} 

Lastly, before we prove Theorem \ref{thm:main} in the case where there are at least two irreducible maps in the generating set $S$, we remind the reader that if $\phi=x^2+c$ for some $c\in\mathbb{Z}\mysetminus\{0,-1\}$, then $\phi^n(0)$ is not a square for all $n\geq2$. In particular, if $\phi$ is irreducible, its adjusted critical orbit contains no squares.      

\begin{lemma}\label{lem-Stoll-irred} Let $\phi=x^2+c$ for some $c\in\mathbb{Z}\mysetminus\{0,-1\}$. Then $\phi^n(0)$ is not a rational square for all $n\geq2$.   
\end{lemma} 
\begin{proof} See \cite[Corollary 1.3]{Stoll} and its proof.  
\end{proof} 

We now have all the tools in place to prove Theorem \ref{thm:main} when $S$ contains at least $2$ irreducible polynomials. 

\begin{proof}[(Proof of Theorem \ref{thm:main})]  Let $S=\{x^2+c_1,\dots, x^2+c_s\}$ for some distinct $c_i\in\mathbb{Z}$, and suppose that $S$ contains at least $2$ irreducible polynomials; up to reordering, we may call them $\phi_1=x^2+c_1$ and $\phi_2=x^2+c_2$. From here, we proceed in cases. Suppose first that $(c_1,c_2)$ is not an exceptional pair. Then we may apply Proposition \ref{prop:nosquares} and (after reordering if necessary) choose $N$ such that $\phi_1(\phi_2(b))$ is not a rational square for all $b\in\mathbb{Z}$. In particular, it follows from Proposition \ref{prop:stability} and Lemma \ref{lem-Stoll-irred} that 
\[\{\phi_1^N\circ\phi_2\circ F\,:\, F\in M_S\}\] 
is a set of irreducible polynomials (a stronger claim than the one made in the statement of the theorem): the first $N$-terms of the adjusted critical orbit of any $f=\phi_1^N\circ\phi_2\circ F$ are not squares by Lemma  \ref{lem-Stoll-irred} and the remaining terms are of the form $\phi_1(\phi_2(b))$ for some $b\in\mathbb{Z}$.

Therefore, we may assume that $(c_1,c_2)$ is an exceptional pair. Then clearly, $(c_1,c_2)\neq (-1,-3)$ as $x^2-1$ is reducible. Hence, (again after reordering if necessary) we may assume that $(c_1,c_2)=(s^2-s^4,-1-s^2-s^4)$ for some $s\in\mathbb{Z}\mysetminus\{0,\pm{1}\}$. Then Proposition \ref{prop:nosquares-exceptional} implies that there is an $N$ such that $\phi_1^N\circ\phi_2\circ\phi_1(b)$ is not a rational square for all $b\in\mathbb{Z}$. In this case, 
\[\{\phi_1^N\circ\phi_2\circ\phi_1\circ F\,:\, F\in M_S\}\]
is a set of irreducible polynomials by Proposition \ref{prop:stability}: the first $N$-terms of the adjusted critical orbit of any $f=\phi_1^N\circ\phi_2\circ F$ are not squares by Lemma  \ref{lem-Stoll-irred}, the next term is $\phi_1^N\circ\phi_2(0)=\phi_1^N(-1-s^2-s^4)$, which is not a square by Lemma \ref{lem:Per-square} and Proposition \ref{prop:portraits} (since it is easy to check that $-1-s^2-s^4\not\in\{\pm{s^2},\pm{(1-s^2)}\}$ for all $s\neq0$), and the remaining terms are of the form $\phi_1^N\circ\phi_2\circ\phi_1(b)$ for some $b\in\mathbb{Z}$.                
\end{proof} 

\section{Classification of Exceptional Pairs} \label{sec:exceptionalpairs}
In this section, we prove Theorem \ref{thm:classification} from the introduction, classifying the exceptional pairs of polynomials $(x^2+c_1,x^2+c_2)$; this means both $\phi_1=x^2+c_1$ and $\phi_2=x^2+c_2$ have periodic points that are rational squares and that both $\phi_1(\mathbb{Z})\cap\PrePer(\phi_2,\mathbb{Q})$ and $\phi_2(\mathbb{Z})\cap\PrePer(\phi_1,\mathbb{Q})$ are non-empty sets. In particular, if $(c_1,c_2)$ are exceptional then Proposition \ref{prop:portraits} implies that there are $x,y,s,t\in\mathbb{Z}$ such that \vspace{.05cm}  
\begin{equation}\label{exceptionalproperties}
\begin{split} 
c_1&=\,s^2-s^4,\;-1-s^2-s^4 \\[5pt] 
c_2&=t^2-t^4,\;-1-t^2-t^4 \\[5pt] 
x^2+c_1&=\pm{t^2},\; \pm{(t^2-1)},\; \pm{(t^2+1)}\\[5pt] 
 y^2+c_2&=\pm{s^2},\; \pm{(1-s^2)},\; \pm{(1+s^2)}\\[2pt]
 \end{split} 
\end{equation} 
The case when $c=s^2-s^4$ corresponds to a square fixed-point and the case when $c=-1-s^2-s^4$ corresponds to a square in a 2-cycle. Therefore, up to reordering $c_1$ and $c_2$, we can break the classification of the integral solutions to \eqref{exceptionalproperties} into three main subcases: both $\phi_1$ and $\phi_2$ have square fixed points, $\phi_1$ has a square fixed point and $\phi_2$ has a square in a $2$-cycle, and both $\phi_1$ and $\phi_2$ have a square in a $2$-cycle. Moreover, each of these subcases are the topic of a subsection within this section. In total there are $48$ cases to be checked (modulo symmetry), and each is assigned a lemma below. Luckily, the same few techniques: Runge's method, reduction modulo $4$ or $8$, finding the integral points on some hyperelliptic quartics, and explicitly describing the Mordell Weil groups of some rank-zero elliptic curves, settle all of these problems.     
\subsection{Two Fixed Points} 
In this subsection, we assume that both $\phi_1$ and $\phi_2$ have square fixed points. Hence, $c_1=s^2-s^4$ and $c_2=t^2-t^4$ for some $s,t\in\mathbb{Z}$. Moreover, we assume that both $\phi_1(\mathbb{Z})\cap\PrePer(\phi_2,\mathbb{Q})$ and $\phi_2(\mathbb{Z})\cap\PrePer(\phi_1,\mathbb{Q})$ are non-empty sets. Therefore, Proposition \ref{prop:portraits} implies that 
\[x^2+s^2-s^4=\pm{t^2},\; \pm{(t^2-1)} \;\;\;\text{and}\;\;\; y^2+t^2-t^4=\pm{s^2},\; \pm{(s^2-1)}\]
for some $x,y\in\mathbb{Z}$. From here, we handle each of these $16$ cases (corresponding to a pair of preperiodic points from $\phi_1$ and $\phi_2$) in subsequent lemmas.     
\begin{lemma}\label{lem:case1.1} Let $x,y,s,t\in\mathbb{Z}$ be such that  
\begin{equation}\label{eq:case1.1}
x^2+s^2-s^4=t^2\qquad\text{and}\qquad y^2+t^2-t^4=s^2.
\end{equation} 
Then $(x,y,s,t)=(\pm{1},0,0,\pm{1})$, $(0,\pm{1},\pm{1},0)$, or $(\pm{u^2},\pm{u^2},\pm{u},\pm{u})$ for some $u\in\mathbb{Z}$.
\end{lemma}
\begin{proof} Suppose first that $s=0$. Then the equation on the right in \eqref{eq:case1.1} implies that 
\[y^2=t^2(t^2-1).\]
In particular, $t^2-1=z^2$ for some $z\in\mathbb{Z}$. Hence, $(t-z)(t+z)=1$. However, this implies that either $t-z=1=t+z$ or $t-z=-1=t+z$. Hence, $t=\pm{1}$ in either case. However, substituting in $t=\pm{1}$ and $s=0$ into the equation on the right of \eqref{eq:case1.1} implies that $y=0$. From here the left equation in \eqref{eq:case1.1} implies that $x=\pm{1}$. To summarize: we have shown that if $s=0$, then $(x,y,s,t)=(\pm{1},0,0,\pm{1})$. On the other hand, if we replace $(x,s)\leftrightarrow (y,t)$ and repeat the argument above, then we may similarly deduce that if $t=0$, then $(x,y,s,t)=(0,\pm{1},\pm{1},0)$. In particular, we may assume that neither $s$ nor $t$ is zero. Moreover, replacing $(s,t)$ with $(\pm{s},\pm{t})$ if need be, we may assume that $s$ and $t$ are both positive. Now suppose for a contradiction that $t^2>s^2\geq1$. Then $-t^2+1<0<s^2$ and \vspace{.1cm}  
\begin{equation*}
\begin{split}  
(t^2-1)^2&=t^4-t^2+(-t^2+1)\\[5pt]
&<t^4-t^2+s^2=(t^2)^2-(t^2-s^2)\\[5pt] 
&<(t^2)^2.
\end{split} 
\end{equation*} 
Therefore, the quantity $t^4-t^2+s^2$ is strictly between two consecutive integer squares and so cannot be a square itself. However, this contradicts the equation on the right of \eqref{eq:case1.1}. Moreover, a similar argument (swapping $s$ and $t$) implies that $s^2>t^2\geq1$ is impossible. Therefore, if $s$ and $t$ are non-zero, then $s^2=t^2$. In particular, we deduce that $(x,y,s,t)=(\pm{u^2},\pm{u^2},\pm{u},\pm{u})$ for some $u\in\mathbb{Z}$ in this case as claimed.            
\end{proof}   

\begin{lemma}\label{lem:case1.2} Let $x,y,s,t\in\mathbb{Z}$ be such that  
\begin{equation}\label{eq:case1.2}
x^2+s^2-s^4=t^2\qquad\text{and}\qquad y^2+t^2-t^4=-s^2.
\end{equation} 
Then $(x,y,s,t)=(\pm{1},0,0,\pm{1})$ or $(0,0,0,0)$.
\end{lemma}
\begin{proof} 
If $s=0$, then $x^2=t^2$ and $y^2=t^2(t^2-1)$. In particular, $t^2-1$ is a square so that $t=\pm{1}$; see the proof of Lemma \ref{lem:case1.1} above. Therefore, $(x,y,s,t)=(\pm{1},0,0,\pm{1})$ in this case. Likewise if $t=0$, then $y^2=-s^2$, which immediately implies that $y=s=0$. Therefore, $(x,y,s,t)=(0,0,0,0)$ in this case. Hence, we may assume that neither $s$ nor $t$ is zero. Now suppose first that $s^2>t^2\geq1$. Then, similar to the proof of Lemma \ref{lem:case1.1}, we have that 
\begin{equation*} 
\begin{split}
(s^2-1)^2&=s^4-s^2-s^2+1\\[5pt] 
&<s^4-s^2+t^2\\[5pt]
&=s^4-(s^2-t^2)\\[5pt] 
&<(s^2)^2.   
\end{split} 
\end{equation*}
Hence, $s^4-s^2+t^2$ is strictly between two consecutive integer squares and so cannot be a square itself. However, this contradicts the equation on the left of \eqref{eq:case1.2}. Suppose next that $t^2>s^2\geq1$. Note that if $t^2-s^2=1$, then either $t-s=1=t+s$ or $t-s=-1=t+s$. But in either case, this implies that $s=0$, a contradiction. Therefore, we may assume that $t^2-s^2>1$. Hence, $-t^2+1<-s^2$ and  
\begin{equation*}\label{eq:nosolutions2} 
\begin{split}
(t^2-1)^2&=t^4-t^2+(-t^2+1)\\[5pt] 
&<t^4-t^2-s^2=t^4-(t^2+s^2)\\[5pt]
&<t^4=(t^2)^2.   
\end{split} 
\end{equation*}
Therefore, $t^4-t^2-s^2$ is strictly between two consecutive integer squares and so cannot be a square itself. However, this contradicts the equation on the right of \eqref{eq:case1.2}. In particular, if $s$ and $t$ are non-zero, then $s^2=t^2$. But then the right side of of \eqref{eq:case1.2} implies that 
\begin{equation}\label{eq:nosolutions3} 
y^2=t^4-t^2-s^2=t^4-2t^2=t^2(t^2-2),
\end{equation} 
so that $t^2-2=z^2$ for some $z\in\mathbb{Z}$ (since $t\neq0$). However, this equation has no solutions modulo $4$. Therefore, we have only the solutions corresponding to $s=0$ and $t=0$ already listed above.          
\end{proof} 

\begin{lemma}\label{lem:case1.3} Let $x,y,s,t\in\mathbb{Z}$ be such that  
\begin{equation}\label{eq:case1.3}
x^2+s^2-s^4=t^2\qquad\text{and}\qquad y^2+t^2-t^4=1-s^2.
\end{equation} 
Then $(x,y,s,t)=(\pm{1},\pm{1},0,\pm{1})$, $(0,0,\pm{1},0)$ or $(\pm{u^2},\pm{(u^2-1)},\pm{u},\pm{u})$ for some $u\in\mathbb{Z}$.
\end{lemma}
%$(\pm{1},\pm{1},0,\pm{1})$
\begin{proof} 
Suppose first that $s=0$, so that $y^2=t^4-t^2+1$. However, the hyperelliptic curve $C$ with affine model $y^2=t^4-t^2+1$ is birational over $\mathbb{Q}$ to the elliptic curve $E: Y^2 = X^3 + 2X^2 - 3X$, and we compute with Magma that $E(\mathbb{Q})\cong\mathbb{Z}/2\mathbb{Z}\times\mathbb{Z}/4\mathbb{Z}$. Hence $\#C(\mathbb{Q})=8$, and excluding the two points at infinity of $C$, we deduce that $(t,y)=\{(\pm{1},\pm{1}), (0,\pm{1})\}$. In particular, we obtain the points $(\pm{1},\pm{1},0,\pm{1})$ and the points $(0,\pm{1},0,0)=(\pm{u^2},\pm{(u^2-1)},\pm{u},\pm{u})$ for $u=0$ listed above. Likewise, if $t=0$, then $x^2=s^2(s^2-1)$. Moreover, since we have classified the points with $s$-coordinate zero, we may assume that $s\neq0$. Thus $s^2-1=z^2$ for some $z\in\mathbb{Z}$. In particular, $s-z=1=s+z$ or $s-z=-1=s+z$. In either case, $z=0$ and $s=\pm{1}$. But then we obtain the points $(x,y,s,t)=(0,0,\pm{1},0)$ listed above. From here we assume that $s$ and $t$ are both non-zero. Now suppose that $t^2>s^2$. Then $t^2+s^2>2s^2\geq2$ and  
\begin{equation*}
\begin{split}
(t^2-1)^2&=t^4-t^2-t^2+1\\[5pt] 
&< t^4-t^2-s^2+1\\[5pt] 
&=t^4-(t^2+s^2-1)\\[5pt] 
&<t^4=(t^2)^2. 
\end{split} 
\end{equation*} 
Hence, $t^4-t^2-s^2+1$ is strictly between consecutive integer squares and so cannot be a square itself. But this contradicts the right side of \eqref{eq:case1.3}. Likewise, if $s^2>t^2>0$ then $s^2-1>0$ and  
\begin{equation*}
\begin{split}
(s^2-1)^2&=(s^4-s^2)-(s^2-1)\\[5pt] 
&< s^4-s^2+t^2\\[5pt] 
&=s^4-(s^2-t^2)\\[5pt] 
&<s^4=(s^2)^2 
\end{split} 
\end{equation*} 
Hence, $s^4-s^2+t^2$ is strictly between consecutive integer squares and so cannot be a square itself. But this contradicts the left side of \eqref{eq:case1.3}. Therefore, $s^2=t^2$ from which it easily follows that $(x,y,s,t)=(\pm{u^2},\pm{(u^2-1)},\pm{u},\pm{u})$ for some $u\in\mathbb{Z}$.  
\end{proof} 

\begin{lemma}\label{lem:case1.4} Let $x,y,s,t\in\mathbb{Z}$ be such that  
\begin{equation}\label{eq:case1.4}
x^2+s^2-s^4=t^2\qquad\text{and}\qquad y^2+t^2-t^4=-(1-s^2).
\end{equation} 
Then $(x,y,s,t)=(\pm{1},0,\pm{1},\pm{1})$ or $(0,0,\pm{1},0)$.
\end{lemma}
\begin{proof} If $s=0$, then $y^2+t^2-t^4=-1$. However, this equation has no solutions $(t,y)\in(\mathbb{Z}/4\mathbb{Z})^2$. Therefore, $s$ is non-zero. On the other hand if $t=0$, then $x^2=s^2(s^2-1)$. In particular (since $s\neq0$), $s^2-1=z^2$ for some $z\in\mathbb{Z}$. But then (as in various proofs above) $s=\pm{1}$ and $z=0$. Hence, we obtain the points $(x,y,s,t)=(0,0,\pm{1},0)$ listed above.  From here we assume that $s$ and $t$ are both non-zero. Now suppose that $t^2>s^2\geq1$. Then  $2<t^2+s^2$ and \vspace{.1cm}  
\begin{equation*}
\begin{split}
(t^2-1)^2&=t^4-t^2+s^2+1-(t^2+s^2) \\[5pt] 
&<t^4-t^2+s^2+1-2\\[5pt]
&=t^4-t^2+s^2-1\\[5pt]
&<t^4-(t^2-s^2)\\[5pt]
&< t^4=(t^2)^2. 
\end{split} 
\end{equation*} 
Hence, $t^4-t^2+s^2-1$ cannot be an integral square, contradicting the right side of \eqref{eq:case1.4}. However, if $s^2>t^2\geq1$ then $1-s^2<0<t^2$ and  
\begin{equation*}
\begin{split} 
(s^2-1)^2&=s^4-s^2+(1-s^2)\\[5pt] 
&<s^4-s^2+t^2\\[5pt] 
&=s^4-(s^2-t^2)\\[5pt] 
&<s^4=(s^2)^2.
\end{split} 
\end{equation*}
Hence, $s^4-s^2+t^2$ cannot be an integral square, contradicting the left side of \eqref{eq:case1.4}. Therefore, we deduce that $s^2=t^2$. However, then the right side of \eqref{eq:case1.4} implies that $y^2=t^4-1$. However, the hyperelliptic curve with affine model $y^2=t^4-1$ is birational over $\mathbb{Q}$ to the elliptic curve $E: Y^2 = X^3 + 4X$, and we compute with Magma that $E(\mathbb{Q})\cong\mathbb{Z}/4\mathbb{Z}$. Hence, $\#C(\mathbb{Q})=4$ and (disregarding the two rational points at infinity) we have that $(t,y)=(\pm{1},0)$. However, we then have that $(x,y,s,t)=(\pm{1},0,\pm{1},\pm{1})$, one of the points listed above.   
\end{proof} 

\begin{lemma}\label{lem:case1.5} Let $x,y,s,t\in\mathbb{Z}$ be such that  
\begin{equation}\label{eq:case1.5}
x^2+s^2-s^4=-t^2\qquad\text{and}\qquad y^2+t^2-t^4=s^2.
\end{equation}
Then $(x,y,s,t)=(0,\pm{1},\pm{1},0)$ or $(0,0,0,0)$.
\end{lemma}  
\begin{proof}
Substitute $(x,y,s,t)\rightarrow (y,x,t,s)$ and apply Lemma \ref{lem:case1.2}. 
\end{proof} 

\begin{lemma}\label{lem:case1.6} Let $x,y,s,t\in\mathbb{Z}$ be such that  
\begin{equation}\label{eq:case1.6}
x^2+s^2-s^4=-t^2\qquad\text{and}\qquad y^2+t^2-t^4=-s^2.
\end{equation} 
Then $(x,y,s,t)=(0,0,0,0)$.
\end{lemma}
\begin{proof}
Suppose first that $s=0$. Then $x^2=-t^2$ so that $x=0=t$. But then $y=0$ also by the right side of \eqref{eq:case1.6}. Therefore, $(x,y,s,t)=(0,0,0,0)$ in this case. Moreover, the same argument swapping $s$ and $t$, shows that $(x,y,s,t)=(0,0,0,0)$ whenever $t=0$. Therefore, we may assume that both $s$ and $t$ are non-zero. Suppose for a contradiction that $t^2>s^2\geq1$. Note that if $t^2-s^2=1$, then either $t-s=1=t+s$ or $t-s=-1=t+s$. But in either case, this implies that $s=0$, a contradiction. Therefore, we may assume that $t^2-s^2>1$. Hence, $-t^2+1<-s^2$ and  
\begin{equation*} 
\begin{split}
(t^2-1)^2&=t^4-t^2+(-t^2+1)\\[5pt]
&<t^4-t^2-s^2=t^4-(t^2+s^2)\\[5pt]
&<t^4=(t^2)^2  
\end{split}  
\end{equation*} 
Therefore, $t^4-t^2-s^2$ cannot be an integral square, contradicting the right side of \eqref{eq:case1.6}. However, these equations are symmetric in $s$ and $t$, and therefore $s^2>t^2$  is also impossible. Hence, for non-zero $s$ and $t$, it must be the case that $s^2=t^2$. But then $x^2=s^4-2s^2=s^2(s^2-2)$, which implies that $s^2-2=z^2$ for some $z\in\mathbb{Z}$ (since $s$ is not zero). But this equation has no solutions mod $4$. Therefore, $s=t=0$ gives the only possible solution. 
\end{proof} 

\begin{lemma}\label{lem:case1.7} Let $x,y,s,t\in\mathbb{Z}$ be such that  
\begin{equation}\label{eq:case1.7}
x^2+s^2-s^4=-t^2\qquad\text{and}\qquad y^2+t^2-t^4=1-s^2.
\end{equation} 
Then $(x,y,s,t)=(0,\pm{1},0,0)$ or $(0,0,\pm{1},0)$. 
\end{lemma}
\begin{proof} If $t=0$, then $y^2+s^2=1$. Hence, $(s,y)=(0,\pm{1})$ or $(\pm{1},0)$ and we obtain $(x,y,s,t)=(0,\pm{1},0,0)$ or $(0,0,\pm{1},0)$ listed above. Likewise if $s=0$, then $x^2=-t^2$ and $x=0=t$. But then $(x,y,s,t)=(0,\pm{1},0,0)$, and we obtain a solution already recorded. Therefore, we may assume that neither $s$ nor $t$ is zero. Now suppose that $t^2>s^2\geq1$. Then $t^2+s^2\geq2$ and 
\begin{equation*}
\begin{split}
(t^2-1)^2&=t^4-t^2-t^2+1\\[5pt] 
&<t^4-t^2-s^2+1\\[5pt]
&=t^4-(t^2+s^2-1)\\[5pt]
&< t^4=(t^2)^2. 
\end{split} 
\end{equation*} 
Hence, $t^4-t^2-s^2+1$ cannot be an integral square, contradicting the right side of \eqref{eq:case1.7}. Likewise, assume that $s^2>t^2$. Note that if $s^2-t^2=1$, then either $s-t=1=s+t$ or $s-t=-1-s+t$. But in either case, $t$ must be zero, a contradiction. Hence, $s^2-t^2>1$ so that $-s^2+1<-t^2$. However, we then have that   
\begin{equation*}
\begin{split} 
(s^2-1)^2&=s^4-s^2+(-s^2+1)\\[5pt] 
&<s^4-s^2-t^2\\[5pt] 
&=s^4-(s^2+t^2)\\[5pt] 
&<s^4=(s^2)^2.
\end{split} 
\end{equation*}
Hence, $s^4-s^2-t^2$ cannot be an integral square, contradicting the left side of \eqref{eq:case1.7}. Therefore, it must be the case that $s^2=t^2$. But in this case, the left side of \eqref{eq:case1.7} becomes $x^2=s^2(s^2-2)$. In particular, since $s\neq0$ we have that $s^2-2=z^2$ for some $z\in\mathbb{Z}$. However, this equation has no solutions modulo $4$. Therefore, we have found all solutions.         
\end{proof} 

\begin{lemma}\label{lem:case1.8} Let $x,y,s,t\in\mathbb{Z}$ be such that  
\begin{equation}\label{eq:case1.8}
x^2+s^2-s^4=-t^2\qquad\text{and}\qquad y^2+t^2-t^4=-(1-s^2).
\end{equation} 
Then $(x,y,s,t)=(0,0,\pm{1},0)$. 
\end{lemma}
\begin{proof} If $t=0$, then $y^2-s^2=-1$. Then either $y-s=-1$ and $y+s=1$ or $y-s=1$ and $y+s=-1$. However, in either case $y=0$ and $s=\pm{1}$. Hence, we obtain the solutions $(x,y,s,t)=(0,0,\pm{1},0)$ listed above. On the other hand, if $s=0$ then $x^2=-t^2$ and so $x=0=t$. But this implies $y^2=-1$, a contradiction. Therefore, we assume that $s$ and $t$ are both non-zero. Assume for a contradiction that $s^2>t^2\geq1$. Note that if $s^2-t^2=1$, then (as in several earlier proofs) $t=0$. Hence, $s^2-t^2>1$ and thus $-s^2+1<-t^2$ and
\begin{equation*}
\begin{split} 
(s^2-1)^2&=s^4-s^2+(-s^2+1)\\[5pt] 
&<s^4-s^2-t^2\\[5pt] 
&=s^4-(s^2+t^2)\\[5pt] 
&<s^4=(s^2)^2.
\end{split} 
\end{equation*}
Hence, $s^4-s^2-t^2$ cannot be an integral square, contradicting the left side of \eqref{eq:case1.8}. Similarly, if $t^2>s^2\geq1$ then $2<t^2+s^2$ and \vspace{.1cm}  
\begin{equation*}
\begin{split}
(t^2-1)^2&=t^4-t^2+s^2+1-(t^2+s^2) \\[5pt] 
&<t^4-t^2+s^2+1-2\\[5pt]
&=t^4-t^2+s^2-1\\[5pt]
&<t^4-(t^2-s^2)\\[5pt]
&< t^4=(t^2)^2. 
\end{split} 
\end{equation*} 
Hence, $t^4-t^2+s^2-1$ cannot be an integral square, contradicting the right side of \eqref{eq:case1.8}. Therefore, we may assume that $s^2=t^2$. But then $x^2=s^2(s^2-2)$ and so $s^2-2=z^2$ for some $z\in\mathbb{Z}$ (since $s\neq0$). However, this equation has no solutions mod $4$. Thus $(x,y,s,t)=(0,0,\pm{1},0)$ are the only integral solutions as claimed.  
\end{proof} 

\begin{lemma}\label{lem:case1.9} Let $x,y,s,t\in\mathbb{Z}$ be such that  
\begin{equation}\label{eq:case1.9}
x^2+s^2-s^4=1-t^2\qquad\text{and}\qquad y^2+t^2-t^4=s^2.
\end{equation} 
Then $(x,y,s,t)=(\pm{1},\pm{1},\pm{1},0)$, $(0,0,0,\pm{1})$ or $(\pm{(u^2-1)},\pm{u^2},\pm{u},\pm{u})$ for some $u\in\mathbb{Z}$.
\end{lemma}
\begin{proof} Substitute $(x,y,s,t)\rightarrow (y,x,t,s)$ and apply Lemma \ref{lem:case1.3}.   
\end{proof} 

\begin{lemma}\label{lem:case1.10} Let $x,y,s,t\in\mathbb{Z}$ be such that  
\begin{equation}\label{eq:case1.10}
x^2+s^2-s^4=1-t^2\qquad\text{and}\qquad y^2+t^2-t^4=-s^2.
\end{equation} 
Then $(x,y,s,t)=(\pm{1},0,0,0)$ or $(0,0,0,\pm{1})$. 
\end{lemma}
\begin{proof} Substitute $(x,y,s,t)\rightarrow (y,x,t,s)$ and apply Lemma \ref{lem:case1.7}.   
\end{proof} 

\begin{lemma}\label{lem:case1.11} Let $x,y,s,t\in\mathbb{Z}$ be such that  
\begin{equation}\label{eq:case1.11}
x^2+s^2-s^4=1-t^2\qquad\text{and}\qquad y^2+t^2-t^4=1-s^2.
\end{equation} 
Then $(x,y,s,t)=(0,\pm{1},0,\pm{1})$, $(\pm{1},0,\pm{1},0)$ or $(\pm{(u^2-1)}, \pm{(u^2-1)}, \pm{u},\pm{u})$ for $u\in\mathbb{Z}$. 
\end{lemma}
\begin{proof} Suppose that $s=0$, then $x^2+t^2=1$ and $(x,t)\in\{(0,\pm{1}),(\pm{1},0)\}$. In the first case, $(x,y,s,t)=(0,\pm{1},0,\pm{1})$ listed above. In the second case, $(x,y,s,t)=(\pm{1},\pm{1},0,0)$ corresponding to $u=0$ above. This problem is symmetric in $s$ and $t$, and so repeating this process with $t=0$ we obtain the new point $(x,y,s,t)=(\pm{1},0,\pm{1},0)$ listed above. Therefore, we assume that $s$ and $t$ both are non-zero. Suppose for a contradiction that $s^2>t^2\geq1$. Then $s^2+t^2>2$ and
\begin{equation*}
\begin{split} 
(s^2-1)^2&=s^4-s^2-s^2+1\\[5pt] 
&<s^4-s^2-t^2+1\\[5pt] 
&=s^4-(s^2+t^2-1)\\[5pt] 
&<s^4=(s^2)^2.
\end{split} 
\end{equation*}
Hence, $s^4-s^2-t^2+1$ cannot be an integral square, contradicting the left side of \eqref{eq:case1.11}. But again, this problem is symmetric in $s$ and $t$, so we may deduce that $s^2=t^2$, from which the points $(x,y,s,t)=(\pm{(u^2-1)}, \pm{(u^2-1)}, \pm{u},\pm{u})$ for $u\in\mathbb{Z}$ my be easily obtained.    
\end{proof} 

\begin{lemma}\label{lem:case1.12} Let $x,y,s,t\in\mathbb{Z}$ be such that  
\begin{equation}\label{eq:case1.12}
x^2+s^2-s^4=1-t^2\qquad\text{and}\qquad y^2+t^2-t^4=-(1-s^2).
\end{equation} 
Then $(x,y,s,t)=(0,0,\pm{1},\pm{1})$ or $(\pm{1},0,\pm{1},0)$. 
\end{lemma}
\begin{proof} If $t=0$, then $y^2-s^2=-1$ and so $y=0$ and $s=\pm{1}$ (by the same argument used at various points above). Hence, we obtain the solutions $(x,y,s,t)=(\pm{1},0,\pm{1},0)$ listed above. On the other hand, if $s=0$, then $x^2+t^2=1$ and $(x,t)\in\in\{(0,\pm{1}),(0,\pm{1})\}$. However, in all of these scenarios $t^2-t^4=0$, so that $y^2=-1$, a contradiction. Therefore, there are no solutions in this case. From here we may assume that $s$ and $t$ are both non-zero. Now assume for a contradiction that $s^2>t^2\geq1$. Then $s^2+t^2>2$ and \vspace{.1cm} 
\begin{equation*}
\begin{split} 
(s^2-1)^2&=s^4-s^2-s^2+1\\[5pt] 
&<s^4-s^2-t^2+1\\[5pt] 
&=s^4-(s^2+t^2-1)\\[5pt] 
&<s^4=(s^2)^2.
\end{split} 
\end{equation*}
Hence, $s^4-s^2-t^2+1$ cannot be an integral square, contradicting the left side of \eqref{eq:case1.12}. Likewise, if $t^2>s^2\geq1$, then $2<t^2+s^2$ and \vspace{.1cm}  
\begin{equation*}
\begin{split}
(t^2-1)^2&=t^4-t^2+s^2+1-(t^2+s^2) \\[5pt] 
&<t^4-t^2+s^2+1-2\\[5pt]
&=t^4-t^2+s^2-1\\[5pt]
&<t^4-(t^2-s^2)\\[5pt]
&< t^4=(t^2)^2. 
\end{split} 
\end{equation*} 
Hence, $t^4-t^2+s^2-1$ cannot be an integral square, contradicting the right side of \eqref{eq:case1.12}. Therefore, we may assume that $s^2=t^2$. But, in this case the right side of \eqref{eq:case1.12} becomes $y^2=t^4-1$. Moreover, we have already determined the integral (in fact rational points) on this hyperelliptic curve in Lemma \ref{lem:case1.4} above: $(y,t)\in\{(0,\pm{1}\}$. Therefore, we obtain the only other solutions $(x,y,s,t)=(0,0,\pm{1},\pm{1})$ listed above.     
 \end{proof}  

\begin{lemma}\label{lem:case1.13} Let $x,y,s,t\in\mathbb{Z}$ be such that  
\begin{equation}\label{eq:case1.13}
x^2+s^2-s^4=-(1-t^2)\qquad\text{and}\qquad y^2+t^2-t^4=s^2.
\end{equation} 
Then $(x,y,s,t)=(0,\pm{1},\pm{1},\pm{1})$ or $(0,0,0,\pm{1})$ 
\end{lemma}
\begin{proof} Substitute $(x,y,s,t)\rightarrow (y,x,t,s)$ and apply Lemma \ref{lem:case1.4}. 
\end{proof} 

\begin{lemma}\label{lem:case1.14} Let $x,y,s,t\in\mathbb{Z}$ be such that  
\begin{equation}\label{eq:case1.14}
x^2+s^2-s^4=-(1-t^2)\qquad\text{and}\qquad y^2+t^2-t^4=-s^2.
\end{equation} 
Then $(x,y,s,t)=(0,0,0,\pm{1})$. 
\end{lemma}
\begin{proof} Substitute $(x,y,s,t)\rightarrow (y,x,t,s)$ and apply Lemma \ref{lem:case1.8}. 
\end{proof} 

\begin{lemma}\label{lem:case1.15} Let $x,y,s,t\in\mathbb{Z}$ be such that  
\begin{equation}\label{eq:case1.15}
x^2+s^2-s^4=-(1-t^2)\qquad\text{and}\qquad y^2+t^2-t^4=1-s^2.
\end{equation} 
$(x,y,s,t)=(0,0,\pm{1},\pm{1})$ or $(0,\pm{1},0,\pm{1})$. 
\end{lemma}
\begin{proof} Substitute $(x,y,s,t)\rightarrow (y,x,t,s)$ and apply Lemma \ref{lem:case1.12}. 
\end{proof} 

\begin{lemma}\label{lem:case1.16} Let $x,y,s,t\in\mathbb{Z}$ be such that  
\begin{equation}\label{eq:case1.16}
x^2+s^2-s^4=-(1-t^2)\qquad\text{and}\qquad y^2+t^2-t^4=-(1-s^2).
\end{equation} 
$(x,y,s,t)=(0,0,\pm{1},\pm{1})$
\end{lemma}
\begin{proof} If $s=0$, then $y^2=t^4-t^2-1$. However, this equation has no solutions mod $4$. Therefore, $s\neq0$. On the other hand, this system of equations is symmetric in $s$ and $t$ and so $t\neq0$ also. Now assume for a contradiction that $t^2>s^2\geq1$. Then $2<t^2+s^2$ and \vspace{.1cm}  
\begin{equation*}
\begin{split}
(t^2-1)^2&=t^4-t^2+s^2+1-(t^2+s^2) \\[5pt] 
&<t^4-t^2+s^2+1-2\\[5pt]
&=t^4-t^2+s^2-1\\[5pt]
&<t^4-(t^2-s^2)\\[5pt]
&< t^4=(t^2)^2. 
\end{split} 
\end{equation*} 
Hence, $t^4-t^2+s^2-1$ cannot be an integral square, contradicting the right side of \eqref{eq:case1.16}. But again, this system of equations is symmetric in $s$ and $t$, so $s^2>t^2$ is also impossible. Therefore, $s^2=t^2$ and $y^2=t^4-1$. However, we have already determined the integral (in fact rational points) on this hyperelliptic curve in Lemma \ref{lem:case1.4} above: $(y,t)\in\{(0,\pm{1}\}$. Therefore, we obtain the only integral solutions $(x,y,s,t)=(0,0,\pm{1},\pm{1})$ listed above.    
\end{proof} 

\subsection{One square fixed point and one square $2$-cycle} 
In this subsection, we assume that $\phi_1$ has a square fixed point and $\phi_2$ has a square in a $2$-cycle. Hence, $c_1=s^2-s^4$ and $c_2=-1-t^2-t^4$ for some $s,t\in\mathbb{Z}$. Moreover, we assume that both $\phi_1(\mathbb{Z})\cap\PrePer(\phi_2,\mathbb{Q})$ and $\phi_2(\mathbb{Z})\cap\PrePer(\phi_1,\mathbb{Q})$ are non-empty sets. Therefore, Proposition \ref{prop:portraits} implies that 
\[x^2+s^2-s^4=\pm{t^2},\; \pm{(t^2+1)} \;\;\;\text{and}\;\;\; y^2-1-t^2-t^4=\pm{s^2},\; \pm{(s^2-1)}\]
for some $x,y\in\mathbb{Z}$. From here, we handle each of these $16$ cases (corresponding to a pair of preperiodic points from $\phi_1$ and $\phi_2$) in subsequent lemmas.

\begin{lemma}\label{lem:case2.1} Let $x,y,s,t\in\mathbb{Z}$ be such that  
\begin{equation}\label{eq:case2.1}
x^2+s^2-s^4=t^2\qquad\text{and}\qquad y^2-1-t^2-t^4=s^2.
\end{equation} 
Then $(x,y,s,t)=(\pm{u}^2,\pm{(u^2+1)},\pm{u},\pm{u})$ for some $u\in\mathbb{Z}$. 
\end{lemma}
\begin{proof} If $s=0$, then the right side of \eqref{eq:case2.1} is the curve, $C: y^2=t^4+t^2+1$. We compute with Magma (using the point $(0,1)$ on $C$) that $C$ is birational over $\mathbb{Q}$ to $E: Y^2=X^3 - 2X^2 - 3X$, an elliptic curve in Weierstrass form. In particular, we compute with Magma that the Mordell-Weil group of $E$ is $E(\mathbb{Q})\cong\mathbb{Z}/2\mathbb{Z}\times \mathbb{Z}/2\mathbb{Z}$. However, the projective model of $C$ has $4$ known rational points: two at infinity and $(0,\pm{1})$. Therefore, we deduce that $(x,y,s,t)=(0,\pm{1},0,0)$, which fits our description above when $u=0$. On the other hand, when $t=0$ we see that $x^2=s^2(s^2-2)$. In particular, if $s\neq0$, then $s^2-2=z^2$ for some $z\in\mathbb{Z}$. However, this equation has no solutions modulo $4$. Therefore, we may assume that $s$ and $t$ are non-zero. Suppose for a contradiction that $s^2>t^2\geq1$. Then  $1-s^2<t^2$ and  
\begin{equation*}
\begin{split}  
(s^2-1)^2&=s^4-s^2-s^2+1\\[5pt]
&<s^4-s^2+t^2=(s^2)^2-(s^2-t^2)\\[5pt] 
&<(s^2)^2.
\end{split} 
\end{equation*} 
Therefore, the quantity $s^4-s^2+t^2$ is strictly between two consecutive integer squares and so cannot be a square itself. However, this contradicts the equation on the left of \eqref{eq:case2.1}. Conversely, suppose that $t^2>s^2\geq1$. Then
\begin{equation*}
\begin{split}  
(t^2+1)^2&=t^4+t^2+t^2+1\\[5pt]
&>t^4+t^2+1+s^2\\[5pt] 
&>t^4=(t^2)^2.
\end{split} 
\end{equation*} 
Therefore, $t^4+t^2+1+s^2$ is strictly between two consecutive integer squares and so cannot be a square itself. However, this contradicts the equation on the right of \eqref{eq:case2.1}. Therefore, we can assume that $s^2=t^2$ and the claim easily follows.          
\end{proof} 

\begin{lemma}\label{lem:case2.2} Let $x,y,s,t\in\mathbb{Z}$ be such that  
\begin{equation}\label{eq:case2.2}
x^2+s^2-s^4=t^2\qquad\text{and}\qquad y^2-1-t^2-t^4=-s^2.
\end{equation} 
Then $(x,y,s,t)\in\{(0,0,\pm{1},0),(0,\pm{1},0,0)\}$. 
\end{lemma}
\begin{proof} Assume first that $t=0$. Then the left side of \eqref{eq:case2.2} implies that $x^2=s^2(s^2-1)$. In particular, either $s=0$ or $s^2-1=z^2$ for some $z\in\mathbb{Z}$. However, when $s=0$ we obtain the points $(x,y,s,t)=(0,\pm{1},0,0)$. On the other hand, we see that $(s-z)(s+z)=1$. In particular, we deduce that $s=\pm{1}$ in this case and obtain the two points $(x,y,s,t)=(0,0,\pm{1},0)$. Therefore, we may assume that $t$ is non-zero. Now, suppose that $s^2>t\geq1$. Then $-s^2+1<t^2$ and  
\begin{equation*}
\begin{split}  
(s^2-1)^2&=s^4-s^2-s^2+1\\[5pt]
&<s^4-s^2+t^2=(s^2)^2-(s^2-t^2)\\[5pt] 
&<(s^2)^2.
\end{split} 
\end{equation*}
Hence, $s^4-s^2+t^2$ is strictly between two consecutive integer squares and so cannot be a square itself. However, this contradicts the equation on the left of \eqref{eq:case2.2}. On the other hand, if $s^2\leq t^2$, then 
\begin{equation*}
\begin{split}  
(t^2+1)^2&=t^4+t^2+t^2+1\\[5pt]
&>t^4+t^2+1-s^2\\[5pt]
&=t^4+ (t^2-s^2)+1\\[5pt]  
&>t^4=(t^2)^2.
\end{split} 
\end{equation*}
Therefore, $t^4+t^2+1-s^2$ is strictly between two consecutive integer squares and so cannot be a square itself. This contradicts the right of \eqref{eq:case2.2}. In particular, it must be the case that $t=0$, and we obtain only the listed points as claimed.        
\end{proof} 

\begin{lemma}\label{lem:case2.3} Let $x,y,s,t\in\mathbb{Z}$ be such that  
\begin{equation}\label{eq:case2.3}
x^2+s^2-s^4=t^2\qquad\text{and}\qquad y^2-1-t^2-t^4=1-s^2.
\end{equation} 
Then $(x,y,s,t)=(\pm{1},\pm{2},0,\pm{1})$ or $(0,\pm{1},\pm{1},0)$. 
\end{lemma}
\begin{proof} If $s=0$, then $y^2=t^4+t^2+2$. However, the hyperelliptic curve $C$ with affine model $y^2=t^4+t^2+2$ is birational to the elliptic curve $E: ZY^2 = X^3 - 2ZX^2 - 7Z^2X$ via the map $\psi:C\rightarrow E$ given by $(t,y,z)\rightarrow(2t^2z - 2yz + z^3, 4t^3 - 4ty + 2tz^2, z^3)$. In particular, if $P=(t,y,1)$ is an integral point on $C$, then $\psi(P)$ is an integral point on $E$. Moreover, we compute with Magma that the integral points of $E$ are $(X,Y)\in\{(-1,\pm{2}), (0,0), (4,\pm{2}), (7,\pm{14})\}$. Hence, after applying the inverse of $\psi$ to the seven integral points of $E$ listed above, we deduce that the full set of affine integral points of $C$ are $(t,y)=(\pm{1},\pm{2})$. Therefore, we obtain the points $(x,y,s,t)=(\pm{1},\pm{2},0,\pm{1})$ listed above. Now assume that $t=0$, then $y^2=2-s^2$ and we deduce that $(s,y)=(\pm{1},\pm{1})$. In particular, we obtain the solutions $(x,y,s,t)=(0,\pm{1},\pm{1},0)$ listed above. Now we assume for a contradiction that $s^2>t^2\geq1$. Then $-s^2+1<t^2$ and  
\begin{equation*}
\begin{split}  
(s^2-1)^2&=s^4-s^2-s^2+1\\[5pt]
&<s^4-s^2+t^2=(s^2)^2-(s^2-t^2)\\[5pt] 
&<(s^2)^2.
\end{split} 
\end{equation*}
Hence, $s^4-s^2+t^2$ cannot be a square itself, contradicting the left side of \eqref{eq:case2.3}. On the other hand, if $t^2>s^2\geq1$, then $t^2>1$ and  
\begin{equation*}
\begin{split}
(t^2+1)^2&=t^4+t^2+t^2+1\\[5pt] 
&> t^4+t^2+2\\[5pt]
&> t^4+t^2+2-s^2\\[5pt] 
&=t^4+(t^2-s^2)+2\\[5pt] 
&>t^4=(t^2)^2. 
\end{split} 
\end{equation*} 
Hence, $t^4+t^2+2-s^2$ cannot be an integral square, contradicting the right side of \eqref{eq:case2.3}. Therefore, we may assume that $s^2=t^2$. However, in this case we obtain an equation $y^2=t^4+2$ with no solutions mod $4$. Therefore, there are only the solutions corresponding to $s=0$ and $t=0$ listed above. 
\end{proof} 

\begin{lemma}\label{lem:case2.4} Let $x,y,s,t\in\mathbb{Z}$ be such that  
\begin{equation}\label{eq:case2.4}
x^2+s^2-s^4=t^2\qquad\text{and}\qquad y^2-1-t^2-t^4=-(1-s^2).
\end{equation} 
Then $(x,y,s,t)=(0,0,0,0)$ or $(0,\pm{1},\pm{1},0)$. 
\end{lemma}
\begin{proof} Suppose that $s=0$. Then $y^2=t^2(t^2+1)$, which implies $t=0$ or $t^2+1=z^2$ for some $z\in\mathbb{Z}$. However, the latter case also implies that $t=0$. Therefore, we obtain the solution $(x,y,s,t)=(0,0,0,0)$. On the other hand, if $t=0$ and $s\neq0$, then $x^2=s^2(s^2-1)$ implies that $z^2=s^2-1$ for some $z\in\mathbb{Z}$. However, it is easy to check by factoring that the only integral solutions to this equation are $(s,z)=(\pm{1},0)$. Therefore, we add the solutions $(x,y,s,t)=(0,\pm{1},\pm{1},0)$ in this case. Now then we may assume that neither $s$ nor $t$ is zero. Moreover, assume for a contradiction that $s^2>t^2\geq1$. Then $-s^2+1<t^2$ and  
\begin{equation*}
\begin{split}  
(s^2-1)^2&=s^4-s^2-s^2+1\\[5pt]
&<s^4-s^2+t^2=(s^2)^2-(s^2-t^2)\\[5pt] 
&<(s^2)^2.
\end{split} 
\end{equation*}
Hence, $s^4-s^2+t^2$ cannot be a square itself, contradicting the left side of \eqref{eq:case2.4}. Likewise, if $t^2>s^2\geq1$, then \vspace{.1cm}  
\begin{equation*}
\begin{split}
(t^2+1)^2&=t^4+t^2+t^2+1 \\[5pt] 
&>t^4+t^2+s^2\\[5pt]
&>t^4=(t^2)^2. 
\end{split} 
\end{equation*} 
Hence, $t^4+t^2+s^2$ cannot be an integral square, contradicting the right side of \eqref{eq:case2.4}. Therefore, we may assume that $s^2=t^2$ and both $s$ and $t$ are non-zero. But then $y^2=t^2(t^2+2)$ so that $t^2+2=z^2$ for some $z\in\mathbb{Z}$. However, this equation has no solutions modulo $4$. Therefore, the only solutions integral solutions to \eqref{eq:case2.4} are those listed above.     
\end{proof} 

\begin{lemma}\label{lem:case2.5} Let $x,y,s,t\in\mathbb{Z}$ be such that  
\begin{equation}\label{eq:case2.5}
x^2+s^2-s^4=-t^2\qquad\text{and}\qquad y^2-1-t^2-t^4=s^2.
\end{equation} 
Then $(x,y,s,t)=(0,\pm{1},0,0)$. 
\end{lemma}
\begin{proof} If $s=0$, then $x^2=-t^2$ and so $x=0=t$. In particular, we obtain the solutions $(x,y,s,t)=(0,\pm{1},0,0)$. On the other hand, if $t=0$, then $y^2-1=s^2$. It follows by factoring $(y-s)(y+s)=1$ that $(y,s)=(\pm{1},0)$, and we recapture the solution $(x,y,s,t)=(0,\pm{1},0,0)$. Therefore, we may assume that neither $s$ nor $t$ is zero. Now suppose for a contradiction that $s^2>t^2\geq1$. Note that if $s^2-t^2=1$, then either $s-t=1=s+t$ or $s-t=-1=s+t$. But in either case, this implies that $t=0$, a contradiction. Therefore, we may assume that $s^2-t^2>1$. Hence, $-s^2+1<-t^2$ and  
\begin{equation*}\label{eq:nosolutions2} 
\begin{split}
(s^2-1)^2&=s^4-s^2+(-s^2+1)\\[5pt] 
&<s^4-s^2-t^2=s^4-(s^2+t^2)\\[5pt]
&<s^4=(s^2)^2.   
\end{split} 
\end{equation*}
Therefore, $s^4-s^2-t^2$ cannot be an integral square, contradicting the left side of \eqref{eq:case2.5}. On the other hand, if $t^2>s^2\geq1$, then 
\begin{equation*}
\begin{split}  
(t^2+1)^2&=t^4+t^2+t^2+1\\[5pt]
&>t^4+t^2+1+s^2\\[5pt] 
&>t^4=(t^2)^2.
\end{split} 
\end{equation*} 
Therefore, $t^4+t^2+1+s^2$ cannot be a square, contradicts the right of \eqref{eq:case2.5}. Therefore, we can assume that $s^2=t^2$ and that both $s$ and $t$ are not zero. But in this case, $x^2=s^2(s^2-2)$ implies that $s^2-2=z^2$ for some $z\in\mathbb{Z}$. However, this equation has no solutions modulo $4$. Therefore, the only integral solutions to \eqref{eq:case2.5} are those corresponding to $s=0=t$ already listed above.   
\end{proof} 

\begin{lemma}\label{lem:case2.6} Let $x,y,s,t\in\mathbb{Z}$ be such that  
\begin{equation}\label{eq:case2.6}
x^2+s^2-s^4=-t^2\qquad\text{and}\qquad y^2-1-t^2-t^4=-s^2.
\end{equation} 
Then $(x,y,s,t)=(0,\pm{1},0,0)$ or $(0,0,\pm{1},0)$. 
\end{lemma}
\begin{proof} Suppose that $s=0$. Then $x^2=-t^2$ and thus $x=0=t$; hence, we obtain the solution $(x,y,s,t)=(0,\pm{1},0,0)$ in this case. Likewise, if $t=0$ then $y^2+s^2=1$ and thus $(y,s)\in\{(0,\pm{1}), (\pm{1},0)\}$. In the first case, we obtain the new solutions $(x,y,s,t)=(0,0,\pm{1},0)$; in the second, we rediscover the solutions correspond to $t=0$. Therefore, we may assume that neither $s$ nor $t$ is zero. On the other hand, if $s^2>t^2\geq1$, then the same argument given in the proof of Lemma \ref{lem:case2.5} implies that $s^4-s^2-t^2$ cannot be an integral square, contradicting the left side of \eqref{eq:case2.6}. Likewise, the proof of Lemma \ref{lem:case2.2} implies that $t^2>s^2\geq1$ is impossible: otherwise, $t^4+t^2+1-s^2$ cannot be an integral square, contradicting the right side of \eqref{eq:case2.6}. Hence, we may assume that $s^2=t^2$ and $s$ is not zero. But in this case $x^2=s^2(s^2-2)$, which implies that $s^2-2=z^2$ for some $z\in\mathbb{Z}$, a contradiction; this equation has no solutions modulo $4$. Therefore, we have found all of the integral solutions to \eqref{eq:case2.6}.          
\end{proof} 

\begin{lemma}\label{lem:case2.7} Let $x,y,s,t\in\mathbb{Z}$ be such that  
\begin{equation}\label{eq:case2.7}
x^2+s^2-s^4=-t^2\qquad\text{and}\qquad y^2-1-t^2-t^4=1-s^2.
\end{equation} 
Then $(x,y,s,t)=(0,\pm{1},\pm{1},0)$.   
\end{lemma}
\begin{proof} 
If $s=0$, then $x^2=-t^2$ and thus $x=0=t$. However, this implies that $y^2=2$, a contradiction. On the other hand if $t=0$, then $x^2=s^2(s^2-1)$ and the fact that $s\neq0$ imply that $s^2-1=z^2$ for some $z\in\mathbb{Z}$. In particular, we deduce that $s=\pm{1}$ by factoring $(s-z)(s+z)=1$. Therefore, we obtain the solutions $(x,y,s,t)=(0,\pm{1},\pm{1},0)$ in this case. Hence, we may assume that neither $s$ nor $t$ is zero. On the other hand, if $s^2>t^2\geq1$, then the same argument given in the proof of Lemma \ref{lem:case2.5} implies that $s^4-s^2-t^2$ cannot be an integral square, contradicting the left side of \eqref{eq:case2.7}. Likewise, the proof of Lemma \ref{lem:case2.3} implies that $t^2>s^2\geq1$ is impossible: otherwise, $t^4+t^2+2-s^2$ cannot be an integral square, contradicting the right side of \eqref{eq:case2.7}. Hence, we may assume that $s^2=t^2$ and that $s\neq0$. But then $x^2=s^2(s^2-2)$, which implies that $s^2-2=z^2$ for some $z\in\mathbb{Z}$. However, this equation has no solutions modulo $4$, a contradiction. Therefore, we have found all of the integral solutions to \eqref{eq:case2.7}. 
\end{proof} 

\begin{lemma}\label{lem:case2.8} Let $x,y,s,t\in\mathbb{Z}$ be such that  
\begin{equation}\label{eq:case2.8}
x^2+s^2-s^4=-t^2\qquad\text{and}\qquad y^2-1-t^2-t^4=-(1-s^2).
\end{equation} 
Then $(x,y,s,t)=(0,0,0,0)$ or $(0,\pm{1},\pm{1},0)$.   
\end{lemma}
\begin{proof} If $s=0$, then $x^2=-t^2$ and thus $x=0=t$ and $y^2=0$. In particular, we find the solution $(x,y,s,t)=(0,0,0,0)$ as claimed. On the other hand, when $t=0$ and $s\neq0$ we see that $x^2=s^2(s^2-1)$. Hence, $s^2-1=z^2$ for some $z\in\mathbb{Z}$. However, by factoring $(s-z)(s+z)=1$, we deduce that $(s,z)=(\pm{1},0)$ and obtain the solutions  $(x,y,s,t)=(0,\pm{1},\pm{1},0)$ listed above. Therefore, we may assume that neither $s$ nor $t$ is zero. On the other hand, if $s^2>t^2\geq1$, then the same argument given in the proof of Lemma \ref{lem:case2.5} implies that $s^4-s^2-t^2$ cannot be an integral square, contradicting the left side of \eqref{eq:case2.8}. Likewise, the proof of Lemma \ref{lem:case2.4} implies that $t^2>s^2\geq1$ is impossible: otherwise, $t^4+t^2+s^2$ cannot be an integral square, contradicting the right side of \eqref{eq:case2.8}. Hence, we may assume that $s^2=t^2$ and that $s\neq0$. But then $x^2=s^2(s^2-2)$, which implies that $s^2-2=z^2$ for some $z\in\mathbb{Z}$. However, this equation has no solutions mod $4$, a contradiction. Thus, we have found all integral solutions to \eqref{eq:case2.8}.   
\end{proof}

\begin{lemma}\label{lem:case2.9} Let $x,y,s,t\in\mathbb{Z}$ be such that  
\begin{equation}\label{eq:case2.9}
x^2+s^2-s^4=1+t^2\qquad\text{and}\qquad y^2-1-t^2-t^4=s^2.
\end{equation} 
Then $(x,y,s,t)=(\pm{1},\pm{1},0,0)$. 
\end{lemma}
\begin{proof} If $s=0$, then $y^2=t^4+t^2+1$. However, we showed in Lemma \ref{lem:case2.1} that the hyperelliptic curve $C: y^2=t^4+t^2+1$ has only two affine integral (in fact rational) points $(t,y)=(0,\pm{1})$. In particular, we find the solutions $(x,y,s,t)=(\pm{1},\pm{1},0,0)$ to \eqref{eq:case2.9} listed above. On the other hand, if $t=0$, then $(y-s)(y+s)=1$ implies that $(s,y)=(0,\pm{1})$. Hence, we rediscover the solutions $(x,y,s,t)=(\pm{1},\pm{1},0,0)$ in this case. Therefore, we may assume that both $s$ and $t$ are non-zero. Now assume for a contradiction that $s^2>t^2\geq1$. Note that if $s^2-t^2=1$, then $(s-t)(s+t)=1$. However, this implies that $t=0$, a contradiction. Therefore, $s^2-t^2>1$ and 
\begin{equation*}
\begin{split}  
(s^2-1)^2&=s^4-s^2-s^2+1\\[5pt]
&<s^4-s^2+t^2+1\\[5pt]
&=s^4-(s^2-t^2-1) \\[5pt] 
&<s^4=(s^2)^2.
\end{split} 
\end{equation*}
Hence, $s^4-s^2+t^2+1$ cannot be an integral square, contradicting the left side of \eqref{eq:case2.9}. On the other hand, if $t^2>s^2\geq1$, then the same argument in the proof of Lemma \ref{lem:case2.5} implies that $t^4+t^2+1+s^2$ cannot be an integral square, contradicting the right side of \eqref{eq:case2.9}. Therefore, we may assume that $s^2=t^2$. But in this case $x^2=s^4+1$. However, the hyperelliptic curve B with affine model $x^2=s^4+1$ is birational over $\mathbb{Q}$ to $E: Y^2 = X^3 - 4X$. Moreover, we compute with Magma that $E(\mathbb{Q})\cong\mathbb{Z}/2\mathbb{Z}\times\mathbb{Z}/2\mathbb{Z}$. In particular $\#C(\mathbb{Q})=4$, and hence (excluding the two rational points at infinity) we deduce that $(s,x)=(0,\pm{1})$. Therefore, we again obtain the solutions $(x,y,s,t)=(\pm{1},\pm{1},0,0)$. Moreover, we may conclude that these are the only integral solutions to \eqref{eq:case2.9} as claimed.              
\end{proof} 

\begin{lemma}\label{lem:case2.10} Let $x,y,s,t\in\mathbb{Z}$ be such that  
\begin{equation}\label{eq:case2.10}
x^2+s^2-s^4=1+t^2\qquad\text{and}\qquad y^2-1-t^2-t^4=-s^2.
\end{equation} 
Then $(x,y,s,t)=(\pm{1},\pm{1},0,0)$ or $(\pm{1},0,\pm{1},0)$. 
\end{lemma}
\begin{proof}
If $s=0$, then $(x-t)(x+t)=1$ and thus $(x,t)=(\pm{1},0)$. In particular, we obtain the solutions $(x,y,s,t)=(\pm{1},\pm{1},0,0)$ listed above. On the other hand, if $t=0$ and $s\neq0$, then $y^2+s^2=1$ implies that $(y,s)=(0,\pm{1})$. In particular, we obtain the solutions $(x,y,s,t)=(\pm{1},0,\pm{1},0)$ listed above. Therefore, we may assume that neither $s$ nor $t$ is zero. On the other hand, if $s^2>t^2\geq1$, then the same argument given in the proof of Lemma \ref{lem:case2.9} implies that $s^4-s^2+t^2+1$ cannot be an integral square, contradicting the left side of \eqref{eq:case2.10}. Likewise, the proof of Lemma \ref{lem:case2.2} implies that $t^2>s^2\geq1$ is impossible: otherwise, $t^4+t^2+1-s^2$ cannot be an integral square, contradicting the right side of \eqref{eq:case2.10}. Hence, we may assume that $s^2=t^2$ and $s$ is not zero. But in this case $x^2=s^4+1$, and we determined in the proof of Lemma \ref{lem:case2.9} that $(s,x)=(0,\pm{1})$ are the only affine integral (in fact rational) points on this hyperelliptic curve. Therefore, we rediscover the solutions $(x,y,s,t)=(\pm{1},\pm{1},0,0)$. Moreover, we deduce that our list above contains all integral solutions to \eqref{eq:case2.10} as claimed.           
\end{proof} 

\begin{lemma}\label{lem:case2.11} Let $x,y,s,t\in\mathbb{Z}$ be such that  
\begin{equation}\label{eq:case2.11}
x^2+s^2-s^4=1+t^2\qquad\text{and}\qquad y^2-1-t^2-t^4=1-s^2.
\end{equation} 
Then $(x,y,s,t)=(\pm{1},\pm{1},\pm{1},0)$.  
\end{lemma}
\begin{proof} If $s=0$, then $(x-t)(x+t)=1$ and thus $(x,t)=(\pm{1},0)$. However, this implies that $y^2=2$, a contradiction. Therefore, there are no integral solutions with $s=0$. On the other hand, if $t=0$, then $y^2+s^2=2$. Hence, $(s,y)=(\pm{1},\pm{1})$ and we obtain the total solutions $(x,y,s,t)=(\pm{1},\pm{1},\pm{1},0)$ listed above. In particular, we may assume that neither $s$ nor $t$ is zero. On the other hand, if $s^2>t^2\geq1$, then the same argument given in the proof of Lemma \ref{lem:case2.9} implies that $s^4-s^2+t^2+1$ cannot be an integral square, contradicting the left side of \eqref{eq:case2.11}. Likewise, the proof of Lemma \ref{lem:case2.3} implies that $t^2>s^2\geq1$ is impossible: otherwise, $t^4+t^2+2-s^2$ cannot be an integral square, contradicting the right side of \eqref{eq:case2.11}. Hence, we may assume that $s^2=t^2$ and $s\neq0$. But in this case $x^2=s^4+1$, and we determined in the proof of Lemma \ref{lem:case2.9} that $(s,x)=(0,\pm{1})$ are the only affine integral (in fact rational) points on this hyperelliptic curve. Therefore, there are no such integral solutions $(x,y,s,t)$ in this case, and we deduce that our list above is exhaustive as claimed.     
\end{proof} 

\begin{lemma}\label{lem:case2.12} Let $x,y,s,t\in\mathbb{Z}$ be such that  
\begin{equation}\label{eq:case2.12}
x^2+s^2-s^4=1+t^2\qquad\text{and}\qquad y^2-1-t^2-t^4=-(1-s^2).
\end{equation} 
Then $(x,y,s,t)=(\pm{1},0,0,0)$ or $(\pm{1},\pm{1},\pm{1},0)$.  
\end{lemma}
\begin{proof}  If $s=0$, then $(x-t)(x+t)=1$ and thus $(x,t)=(\pm{1},0)$. However, this implies that $y=0$ as well, so we obtain the solutions $(x,y,s,t)=(\pm{1},0,0,0)$ listed above. On the other hand, if $t=0$, then $x^2=s^4-s^2+1$. However, the hyperlliptic curve $C$ with affine model $x^2=s^4-s^2+1$ is birational over $\mathbb{Q}$ to the elliptic curve $E: Y^2 = X^3 + 2X^2 - 3X$. Moreover, we compute with Magma that $E(\mathbb{Q})\cong\mathbb{Z}/2\mathbb{Z}\times\mathbb{Z}/4\mathbb{Z}$. In particular, $\#C(\mathbb{Q})=8$ and (after excluding the two rational points at infinity) we deduce that $(s,x)\in\{(0,\pm{1}),(\pm{1},\pm{1})\}$. Now then, since we assume $s\neq 0$ as well, we obtain the solutions $(x,y,s,t)=(\pm{1},\pm{1},\pm{1},0)$ listed above. From here, we may assume that neither $s$ nor $t$ is zero. On the other hand, if $s^2>t^2\geq1$, then the same argument given in the proof of Lemma \ref{lem:case2.9} implies that $s^4-s^2+t^2+1$ cannot be an integral square, contradicting the left side of \eqref{eq:case2.12}. Likewise, the proof of Lemma \ref{lem:case2.4} implies that $t^2>s^2\geq1$ is impossible: otherwise, $t^4+t^2+s^2$ cannot be an integral square, contradicting the right side of \eqref{eq:case2.12}. Hence, we may assume that $s^2=t^2$ and that $s\neq0$.  But in this case $x^2=s^4+1$, and we determined in the proof of Lemma \ref{lem:case2.9} that $(s,x)=(0,\pm{1})$ are the only affine integral (in fact rational) points on this hyperelliptic curve. Therefore, there are no such integral solutions $(x,y,s,t)$ in this case, and we deduce that our list above is exhaustive as claimed.     
\end{proof} 

\begin{lemma}\label{lem:case2.13} There are no simultaneous integral solutions $x,y,s,t\in\mathbb{Z}$ to the equations
\begin{equation}\label{eq:case2.13}
x^2+s^2-s^4=-1-t^2\qquad\text{and}\qquad y^2-1-t^2-t^4=s^2.
\end{equation}
\end{lemma}
\begin{proof} There are no solutions $x,y,s,t\in\mathbb{Z}/4\mathbb{Z}$. 
\end{proof}  

\begin{lemma}\label{lem:case2.14} There are no simultaneous integral solutions $x,y,s,t\in\mathbb{Z}$ to the equations
\begin{equation}\label{eq:case2.14}
x^2+s^2-s^4=-1-t^2\qquad\text{and}\qquad y^2-1-t^2-t^4=-s^2.
\end{equation}
\end{lemma} 
\begin{proof} There are no solutions $x,y,s,t\in\mathbb{Z}/4\mathbb{Z}$. 
\end{proof}  

\begin{lemma}\label{lem:case2.15} There are no simultaneous integral solutions $x,y,s,t\in\mathbb{Z}$ to the equations
\begin{equation}\label{eq:case2.15}
x^2+s^2-s^4=-1-t^2\qquad\text{and}\qquad y^2-1-t^2-t^4=1-s^2.
\end{equation}
\end{lemma} 
\begin{proof} There are no solutions $x,y,s,t\in\mathbb{Z}/4\mathbb{Z}$. 
\end{proof}  

\begin{lemma}\label{lem:case2.16} There are no simultaneous integral solutions $x,y,s,t\in\mathbb{Z}$ to the equations
\begin{equation}\label{eq:case2.16}
x^2+s^2-s^4=-1-t^2\qquad\text{and}\qquad y^2-1-t^2-t^4=-(1-s^2).
\end{equation}
\end{lemma} 
\begin{proof} There are no solutions $x,y,s,t\in\mathbb{Z}/4\mathbb{Z}$. 
\end{proof}  

\subsection{Two square 2-cycles} In this subsection, we assume that both $\phi_1$ and $\phi_2$ have $2$-cycles that contain a square. Hence, $c_1=-1-s^2-s^4$ and $c_2=-1-t^2-t^4$ for some $s,t\in\mathbb{Z}$. Moreover, we assume that both $\phi_1(\mathbb{Z})\cap\PrePer(\phi_2,\mathbb{Q})$ and $\phi_2(\mathbb{Z})\cap\PrePer(\phi_1,\mathbb{Q})$ are non-empty sets. Therefore, Proposition \ref{prop:portraits} implies that 
\[x^2-1-s^2-s^4=\pm{t^2},\; \pm{(t^2+1)} \;\;\;\text{and}\;\;\; y^2-1-t^2-t^4=\pm{s^2},\; \pm{(s^2+1)}\]
for some $x,y\in\mathbb{Z}$. From here, we handle each of these $16$ cases (corresponding to a pair of preperiodic points from $\phi_1$ and $\phi_2$) in subsequent lemmas.

\begin{lemma}\label{lem:case3.1} Let $x,y,s,t\in\mathbb{Z}$ be such that  
\begin{equation}\label{eq:case3.1}
x^2-1-s^2-s^4=t^2\qquad\text{and}\qquad y^2-1-t^2-t^4=s^2.
\end{equation} 
Then $(x,y,s,t)=(\pm({u}^2+1),\pm{(u^2+1)},\pm{u},\pm{u})$ for some $u\in\mathbb{Z}$. 
\end{lemma}
\begin{proof} 
Suppose that $s^2>t^2$. Then 
\begin{equation*}
\begin{split}  
(s^2+1)^2&=s^4+s^2+s^2+1\\[5pt]
&>s^4+s^2+1+t^2\\[5pt]
&>s^4=(s^2)^2.
\end{split} 
\end{equation*}
Therefore, $s^4+s^2+1+t^2$ is strictly between two consecutive integer squares and so cannot be a square itself. But this contradicts the left side of \eqref{eq:case3.1}. On the other hand, this problem is symmetric in $s$ and $t$, so that $t^2>s^2$ is also impossible. Therefore, $s^2=t^2$. However, in this case $x^2=(s^2+1)^2$ and $y^2=(t^2+1)^2$ follows from \eqref{eq:case3.1}. Therefore, $(x,y,s,t)=(\pm({u}^2+1),\pm{(u^2+1)},\pm{u},\pm{u})$ for some $u\in\mathbb{Z}$ as claimed.        
\end{proof} 

\begin{lemma}\label{lem:case3.2} Let $x,y,s,t\in\mathbb{Z}$ be such that  
\begin{equation}\label{eq:case3.2}
x^2-1-s^2-s^4=t^2\qquad\text{and}\qquad y^2-1-t^2-t^4=-s^2.
\end{equation} 
Then $(x,y,s,t)=(\pm{1},\pm{1},0,0)$.  
\end{lemma}
\begin{proof} 
%If $s=0$, then $(x-t)(x+t)=1$ and so $(x,t)=(\pm{1},0)$. In particular, we obtain the solutions $(x,y,s,t)=(\pm{1},\pm{1},0,0)$ listed above. Likewise if $t=0$, then $y^2+s^2=1$ so that $(s,y)\in\{(0,\pm{1}), (\pm{1},0)\}$. On the other hand if $s=\pm{1}$, then $x^2=3$, a contradiction. Therefore, we may assume that neither $s$ nor $t$ is zero. 
Suppose for a contradiction that $s^2>t^2$. Then the same argument given in the proof of Lemma \ref{lem:case3.1} implies that $s^4+s^2+1+t^2$ cannot be an integral square, contradicting the left side of \eqref{eq:case3.2}. Likewise, if $t^2>s^2$ then 
\begin{equation*}
\begin{split}  
(t^2+1)^2&=t^4+t^2+t^2+1\\[5pt]
&>t^4+t^2-s^2+1\\[5pt]
&=t^4+(t^2-s^2)+1\\[5pt] 
&>t^4=(t^2)^2.
\end{split} 
\end{equation*}
Hence, $t^4+t^2-s^2+1$ cannot be an integral square, contradicting the right side of \eqref{eq:case3.2}. In particular, we may assume that $s^2=t^2$. But in this case, $y^2=t^4+1$. However, the hyperelliptic curve B with affine model $y^2=t^4+1$ is birational over $\mathbb{Q}$ to $E: Y^2 = X^3 - 4X$. Moreover, we compute with Magma that $E(\mathbb{Q})\cong\mathbb{Z}/2\mathbb{Z}\times\mathbb{Z}/2\mathbb{Z}$. In particular $\#C(\mathbb{Q})=4$, and hence (excluding the two rational points at infinity) we deduce that $(t,y)=(0,\pm{1})$. Therefore, $t=0=s$ and $x=\pm{1}=y$. In particular, the solutions $(x,y,s,t)=(\pm{1},\pm{1},0,0)$ are the only integral solutions to \eqref{eq:case3.2} as claimed.        
\end{proof} 

\begin{lemma}\label{lem:case3.3} There are no simultaneous integral solutions $x,y,s,t\in\mathbb{Z}$ to the equations
\begin{equation}\label{eq:case3.3}
x^2-1-s^2-s^4=t^2\qquad\text{and}\qquad y^2-1-t^2-t^4=1+s^2.
\end{equation} 
\end{lemma}
\begin{proof} There are no solutions $x,y,s,t\in\mathbb{Z}/8\mathbb{Z}$. 
\end{proof}  

\begin{lemma}\label{lem:case3.4} Let $x,y,s,t\in\mathbb{Z}$ be such that  
\begin{equation}\label{eq:case3.4}
x^2-1-s^2-s^4=t^2\qquad\text{and}\qquad y^2-1-t^2-t^4=-(1+s^2).
\end{equation} 
Then $(x,y,s,t)=(\pm({u}^2+1),\pm{u^2},\pm{u},\pm{u})$ for some $u\in\mathbb{Z}$. 
\end{lemma}
\begin{proof} 
Suppose for a contradiction that $s^2>t^2$. Then the same argument given in the proof of Lemma \ref{lem:case3.1} implies that $s^4+s^2+1+t^2$ cannot be an integral square, contradicting the left side of \eqref{eq:case3.4}. Likewise, if $t^2>s^2$ then 
\begin{equation*}
\begin{split}  
(t^2+1)^2&=t^4+t^2+t^2+1\\[5pt]
&>t^4+(t^2-s^2)\\[5pt]
&>t^4=(t^2)^2.
\end{split} 
\end{equation*}
Hence, $t^4+t^2-s^2$ cannot be an integral square, contradicting the right side of \eqref{eq:case3.4}. In particular, we may assume that $s^2=t^2$, from which it easily follows that $(x,y,s,t)=(\pm({u}^2+1),\pm{u^2},\pm{u},\pm{u})$ for some $u\in\mathbb{Z}$ s claimed.   
\end{proof}

\begin{lemma}\label{lem:case3.5} Let $x,y,s,t\in\mathbb{Z}$ be such that  
\begin{equation}\label{eq:case3.5}
x^2-1-s^2-s^4=-t^2\qquad\text{and}\qquad y^2-1-t^2-t^4=s^2.
\end{equation} 
Then $(x,y,s,t)=(\pm{1},\pm{1},0,0)$. 
\end{lemma} 
\begin{proof} Substitute $(x,y,s,t)\rightarrow(y,x,t,s)$ and apply Lemma \ref{lem:case3.2} 
\end{proof} 

\begin{lemma}\label{lem:case3.6} Let $x,y,s,t\in\mathbb{Z}$ be such that  
\begin{equation}\label{eq:case3.6}
x^2-1-s^2-s^4=-t^2\qquad\text{and}\qquad y^2-1-t^2-t^4=-s^2.
\end{equation} 
Then $(x,y,s,t)=(\pm{1},\pm{1},0,0)$. 
\end{lemma} 
\begin{proof} Suppose for a contradiction that $t^2>s^2$. Then the same argument given in the proof of Lemma \ref{lem:case3.2} implies that $t^4+t^2-s^2+1$ cannot be an integral square, contradicting the right side of \eqref{eq:case3.6}. On the other hand, this problem is symmetric in $s$ and $t$, so that $s^2>t^2$ is also impossible. Hence, we may assume that $s^2=t^2$. However, in this case $y^2=t^4+1$, and we already showed in the proof of Lemma \ref{lem:case3.2} that the only integral (in fact rational) solutions to this equation are $(t,y)=(0,\pm{1})$. Therefore, $t=0=s$ and $x=\pm{1}=y$. In particular, the solutions $(x,y,s,t)=(\pm{1},\pm{1},0,0)$ are the only integral solutions to \eqref{eq:case3.6} as claimed.  
\end{proof} 

\begin{lemma}\label{lem:case3.7} Let $x,y,s,t\in\mathbb{Z}$ be such that  
\begin{equation}\label{eq:case3.7}
x^2-1-s^2-s^4=-t^2\qquad\text{and}\qquad y^2-1-t^2-t^4=1+s^2.
\end{equation} 
Then $(x,y,s,t)=(0,\pm{2},0,\pm{1})$. 
\end{lemma}
\begin{proof} If $s=0$, then $x^2+t^2=1$ so that $(x,t)\in\{(0,\pm{1}),(\pm{1},0)\}$. In the the first case, we obtain the total solution $(x,y,s,t)=(0,\pm{2},0,\pm{1})$ as claimed. However, in the other case (when $s=0=0$), we see that $y^2=2$, a contradiction. Likewise, if $t=0$, then $y^2-s^2=2$. However, this equation has no solitons modulo $4$. Therefore, we may assume that neither $s$ nor $t$ is zero. Suppose for a contradiction that $s^2>t^2\geq1$. Then 
\begin{equation*}
\begin{split}  
(s^2+1)^2&=s^4+s^2+s^2+1\\[5pt]
&>s^4+s^2-t^2+1\\[5pt]
&=t^4+(s^2-t^2)+1\\[5pt] 
&>t^4=(t^2)^2.
\end{split} 
\end{equation*}
Hence, $s^4+s^2-t^2+1$ cannot be an integral square, contradicting the left side of \eqref{eq:case3.7}. On the other hand, suppose that $t^2>s^2\geq1$. Note that if $t^2-s^2=1$, then $(t-s)(t+s)=1$. However, this immediately implies that $s=0$, a contradiction. Therefore $t^2-s^2>1$ and 
\begin{equation*}
\begin{split}  
(t^2+1)^2&=t^4+t^2+t^2+1\\[5pt]
&=t^4+t^2+(t^2-s^2)+s^2+1\\[5pt]
&>t^4+t^2+1+s^2+1\\[5pt] 
&>t^4=(t^2)^2.
\end{split} 
\end{equation*}
Hence, $t^4+t^2+1+s^2+1$ cannot be an integral square, contradicting the right side of \eqref{eq:case3.7}. In particular, we may assume that $s^2=t^2$ and that $s\neq0$. But then $x^2=s^4+1$, and we have already shown (in the proof of Lemma \ref{lem:case3.2} with the curve $y^2=t^4+1$) that the only rational solutions to this equation correspond to $s=0$, a contradiction. Therefore, we have found all integral solutions to \eqref{eq:case3.7} as claimed.        
\end{proof} 

\begin{lemma}\label{lem:case3.8} Let $x,y,s,t\in\mathbb{Z}$ be such that  
\begin{equation}\label{eq:case3.8}
x^2-1-s^2-s^4=-t^2\qquad\text{and}\qquad y^2-1-t^2-t^4=-(1+s^2).
\end{equation} 
Then $(x,y,s,t)=(\pm{1},0,0,0)$. 
\end{lemma}
\begin{proof} If $s=0$, then $x^2+t^2=1$ so that $(x,t)\in\{(0,\pm{1}),(\pm{1},0)\}$. However, in the $t=\pm{1}$ case the left side of \eqref{eq:case3.8} becomes $y^2=2$, a contradiction. Therefore $s=0=t$, and we obtain the solutions $(x,y,s,t)=(\pm{1},0,0,0)$ as claimed. On the other hand, when $t=0$ then $x^2=s^4+s^2+1$. However, the hyperelliptic curve $C$ with affine model $x^2=s^4+s^2+1$ is birational over $\mathbb{Q}$ to $E: Y^2=X^3 - 2X^2 - 3X$. Moreover, we compute with Magma that the Mordell-Weil group of $E$ is $E(\mathbb{Q})\cong\mathbb{Z}/2\mathbb{Z}\times \mathbb{Z}/2\mathbb{Z}$. However, the projective model of $C$ has $4$ known rational points: two at infinity and $(0,\pm{1})$. Therefore, we rediscover the solution $(x,y,s,t)=(\pm{1},0,0,0)$ when $t=0$. In particular, we may assume that neither $s$ nor $t$ is zero. On the other hand, if $s^2>t^2$, then the same argument given in the proof of Lemma \ref{lem:case3.7} implies that $s^4+s^2-t^2+1$ cannot be an integral square, contradicting the left side of \eqref{eq:case3.8}. Likewise the same argument given in the proof of Lemma \ref{lem:case3.4} implies that $t^4+t^2-s^2$ cannot be an integral square, contradicting the right side of \eqref{eq:case3.8}. In particular, we may assume that $s^2=t^2$ and that $s\neq0$. But then $x^2=s^4+1$, and we have already shown (in the proof of Lemma \ref{lem:case3.2} with the curve $y^2=t^4+1$) that the only rational solutions to this equation correspond to $s=0$, a contradiction. Therefore, we have found all integral solutions to \eqref{eq:case3.8} as claimed.       
\end{proof} 

\begin{lemma}\label{lem:case3.9} There are no simultaneous integral solutions $x,y,s,t\in\mathbb{Z}$ to the equations 
\begin{equation}\label{eq:case3.9}
x^2-1-s^2-s^4=1+t^2\qquad\text{and}\qquad y^2-1-t^2-t^4=s^2.
\end{equation} 
\end{lemma} 
\begin{proof} There are no solutions $x,y,s,t\in\mathbb{Z}/8\mathbb{Z}$.  
\end{proof}

\begin{lemma}\label{lem:case3.10} Let $x,y,s,t\in\mathbb{Z}$ be such that  
\begin{equation}\label{eq:case3.10}
x^2-1-s^2-s^4=1+t^2\qquad\text{and}\qquad y^2-1-t^2-t^4=-s^2.
\end{equation} 
Then $(x,y,s,t)=(\pm{2},0,\pm{1},0)$. 
\end{lemma}
\begin{proof} Substitute $(x,y,s,t)\rightarrow(y,x,t,s)$ and apply Lemma \ref{lem:case3.7}. 
\end{proof} 

\begin{lemma}\label{lem:case3.11} There are no simultaneous integral solutions $x,y,s,t\in\mathbb{Z}$ to the equations  
\begin{equation}\label{eq:case3.11}
x^2-1-s^2-s^4=1+t^2\qquad\text{and}\qquad y^2-1-t^2-t^4=1+s^2.
\end{equation}
\end{lemma} 
\begin{proof} There are no solutions $x,y,s,t\in\mathbb{Z}/8\mathbb{Z}$.  
\end{proof} 

\begin{lemma}\label{lem:case3.12} There are no simultaneous integral solutions $x,y,s,t\in\mathbb{Z}$ to the equations  
\begin{equation}\label{eq:case3.12}
x^2-1-s^2-s^4=1+t^2\qquad\text{and}\qquad y^2-1-t^2-t^4=-(1+s^2).
\end{equation}
\end{lemma} 
\begin{proof} There are no solutions $x,y,s,t\in\mathbb{Z}/8\mathbb{Z}$.  
\end{proof} 

\begin{lemma}\label{lem:case3.13} Let $x,y,s,t\in\mathbb{Z}$ be such that  
\begin{equation}\label{eq:case3.13}
x^2-1-s^2-s^4=-(t^2+1)\qquad\text{and}\qquad y^2-1-t^2-t^4=s^2.
\end{equation} 
Then $(x,y,s,t)=(\pm{u}^2,\pm{(u^2+1)},\pm{u},\pm{u})$ for some $u\in\mathbb{Z}$. 
\end{lemma}
\begin{proof} Substitute $(x,y,s,t)\rightarrow (y,x,t,s)$ and apply Lemma \ref{lem:case3.4} 
\end{proof} 

\begin{lemma}\label{lem:case3.14} Let $x,y,s,t\in\mathbb{Z}$ be such that  
\begin{equation}\label{eq:case3.14}
x^2-1-s^2-s^4=-(t^2+1)\qquad\text{and}\qquad y^2-1-t^2-t^4=-s^2.
\end{equation} 
Then $(x,y,s,t)=(0,\pm{1},0,0)$. 
\end{lemma}
\begin{proof} 
Suppose first that $s^2>t^2$. Then 
 \begin{equation*}
\begin{split}  
(s^2+1)^2&=s^4+s^2+s^2+1\\[5pt]
&>s^4+s^2\\[5pt]
&\geq s^4+(s^2-t^2)\\[5pt]
&>s^4=(s^2)^2.
\end{split} 
\end{equation*}
Therefore, $s^4+s^2-t^2$ cannot be a square, contradicting the left side of \eqref{eq:case3.14}. Likewise, the same argument given in the proof of Lemma \ref{lem:case3.2} implies that $t^2>s^2$ is impossible: $t^4+t^2-s^2+1$ cannot be an integral square, contradicting the right side of \eqref{eq:case3.14}. Thus it must be the case that $s^2=t^2$. But then $y^2=t^4+1$, and we showed in the proof of Lemma \ref{lem:case3.2} that the only rational solutions to this equation are $(t,y)=(0,\pm{1})$. In particular, we deduce that $(x,y,s,t)=(0,\pm{1},0,0)$ are the only integral solutions to \eqref{eq:case3.14} as claimed.  
\end{proof} 

\begin{lemma}\label{lem:case3.15} There are no simultaneous integral solutions $x,y,s,t\in\mathbb{Z}$ to the equations  
\begin{equation}\label{eq:case3.15}
x^2-1-s^2-s^4=-(t^2+1)\qquad\text{and}\qquad y^2-1-t^2-t^4=s^2+1.
\end{equation} 
\end{lemma}
\begin{proof} There are no solutions $x,y,s,t\in\mathbb{Z}/8\mathbb{Z}$.  
\end{proof} 

\begin{lemma}\label{lem:case3.16} Let $x,y,s,t\in\mathbb{Z}$ be such that  
\begin{equation}\label{eq:case3.16}
x^2-1-s^2-s^4=-t^2-1\qquad\text{and}\qquad y^2-1-t^2-t^4=-(s^2+1).
\end{equation} 
Then $(x,y,s,t)=(\pm{u}^2,\pm{u^2},\pm{u},\pm{u})$ for some $u\in\mathbb{Z}$. 
\end{lemma}
\begin{proof}
Suppose that $s^2>t^2$. Then 
 \begin{equation*}
\begin{split}  
(s^2+1)^2&=s^4+s^2+s^2+1\\[5pt]
&>s^4+s^2\\[5pt]
&\geq s^4+(s^2-t^2)\\[5pt] 
&>s^4=(s^2)^2.
\end{split} 
\end{equation*}
Therefore, $s^4+s^2-t^2$ cannot be a square, contradicting the left side of \eqref{eq:case3.16}. Moreover, this problem is symmetric in $s$ and $t$. Hence, $t^2<s^2$ is also impossible. Thus $s^2=t^2$, from which the description above easily follows.
\end{proof} 

\subsection{Proof of the classification of exceptional pairs} As a result of Lemmas \ref{lem:case1.1} to \ref{lem:case3.16}, if $(x^2+c_1,x^2+c_2)$ is an exceptional pair with $c_1\neq c_2$, then up to reordering, either 
\[(c_1,c_2)=(-1,-3)\;\;\;\text{or}\;\;\; (c_1,c_2)=(s^2-s^4,-1-s^2-s^4)\] 
for some $s\in\mathbb{Z}$ as claimed.

\end{document}